\documentclass[a4paper,11pt,reqno]{amsart} 
\usepackage{amsmath}
\usepackage{amsfonts}
\usepackage[T1]{fontenc}
\usepackage[utf8]{inputenc}
\usepackage[mathscr]{euscript}
\usepackage{verbatim}
\usepackage{amssymb,latexsym}
\usepackage{amsthm}
\usepackage{eufrak}
\usepackage{color}
\usepackage[lmargin=2.5 cm,rmargin=2.5 cm,tmargin=3.5cm,bmargin=2.5cm,paper=a4paper]{geometry}
\DeclareMathOperator{\curl}{curl}

\newtheorem{thm}{Theorem}[section]

\newtheorem{lem}[thm]{Lemma}

\newtheorem{theorem}[thm]{Theorem}

\newtheorem{lemma}[thm]{Lemma}
\newtheorem{proposition}[thm]{Proposition}
\newtheorem{corollary}[thm]{Corollary}

\theoremstyle{remark}
\newtheorem{rem}[thm]{Remark}
\newcommand{\nb}{\nabla}
\newcommand{\ld}{\lambda}
\newcommand{\Int}[2]{\displaystyle{\int_{#1}^{#2}}}

\newcommand{\Ab}{\mathbf {A}}

\newcommand{\R}{\mathbb{R}}
\newcommand{\clr}{\color{red}}

\newcommand{\N}{\mathbb{N}}
\newcommand{\Sum}[2]{\displaystyle{\sum_{#1}^{#2}}}
\newcommand{\norm}[1]{\left\|#1\right\|}

\def\sig#1{\vbox{\hsize=5.5cm
\kern2cm\hrule\kern1ex
\hbox to \hsize{\strut\hfil #1 \hfil}}}
\newcommand\signatures[4]{%
\vspace{3cm}
\hbox to \hsize{\hfil #1, \today\hfil}
\vspace{3cm}
\hbox to \hsize{\quad#2\hfil\hfil #3\quad}
\vspace{3cm}
\hbox to \hsize{\hfil#4\hfil}}
\numberwithin{equation}{section}

\title[Trace asymptotics]{Semi-classical trace asymptotics for  magnetic Schr\"odinger operators with Robin condition}
\author[A. Kachmar]{Ayman Kachmar}
\address[A. Kachmar]{Lebanese University, Department of Mathematics, Hadath, Lebanon}
\email{ayman.kashmar@liu.edu.lb}
\author[M. Nasrallah]{Marwa Nasrallah}
\address[M. Nasrallah]{Lebanese International University, School of Arts and Sciences, Rayak, Lebanon}
\email{marwa.nasrallah@liu.edu.lb}
\begin{document}
\maketitle
\begin{abstract}
We compute the sum and number of eigenvalues for a certain class of
magnetic Schr\"odinger operators in a domain with boundary.
Functions in the domain of the operator satisfy  a (magnetic) Robin
condition. The calculations are valid in the semi-classical
asymptotic limit and the eigenvalues concerned correspond to
eigenstates localized near the boundary of the domain. The formulas
we derive display the influence of the boundary and the boundary
condition and are valid under a  weak regularity assumption of the
boundary function. Our approach relies on three main points:
reduction to the boundary;  construction of boundary coherent
states; handling the boundary term as a surface electric potential
and controlling the errors by various Lieb-Thirring inequalities.
\end{abstract}
\section{Introduction}

Recently, many papers display the influence of the Robin condition
on the spectrum of the  Laplacian. In planar domains, the papers
\cite{GS, LP, Pan} and references therein contain asymptotics of the
principal eigenvalue.  The tunneling effect for planar domains with
corners is discussed in the paper \cite{HP}. In higher dimensions,
the low-lying eigenvalues are studied in \cite{PanP}, where the
effect of the boundary mean curvature is made precise. Trace
semi-classical asymptotics are obtained in \cite{FG}. In all the
aforementioned papers, there is no magnetic field and the function
in the boundary condition is supposed smooth. The new issue
addressed in this paper is that we include a magnetic field and we
do not assume smoothness of the boundary function in \eqref{eq:bc}
below. The discussion in this paper is limited for planar domains.
Extensions to higher dimensions does not seem trivial; \cite{N}
contains results for the Neumann condition in 3D domains.

Let $\Omega\subset\R^{2}$ be an open domain with a smooth $C^3$ and
compact boundary $\Gamma=\partial\Omega$. We suppose that the
boundary $\partial\Omega$ consists of a finite number of connected
components. The domain $\Omega$ is allowed to be an {\it interior}
or {\it exterior} domain. By smoothness of the boundary
$\partial\Omega$, we can define the unit outward normal vector $\nu$
of $\partial\Omega$.

The magnetic field is defined via a vector field (magnetic
potential). Let $A\in C^{2}(\overline\Omega;\R^{2})$. The magnetic
field is
\begin{equation}\label{eq:mf}
B:= {\rm curl}\, A\,.
\end{equation}

Consider a function $\gamma\in L^3(\partial\Omega)$, a number
$\alpha\geq 1/2$ and a parameter $h>0$. The parameter $h$ is called
the {\it semi-classical} parameter and we shall be concerned with the
asymptotic limit of various quantities when the semi-classical
parameter tends to $0$.

The self-adjoint magnetic Schr\"odinger operator
\begin{equation}\label{Shr-op-Gen}
\mathcal{P}^{\alpha,\gamma}_{h,\Omega}=(-ih\nabla+A)^{2},
\end{equation}
with a boundary condition of the third type (Robin condition)
\begin{equation}\label{eq:bc}
\nu\cdot(-ih\nabla+A)u+h^\alpha\gamma\,u=0\quad{\rm on}~\partial\Omega\,,
\end{equation}
can be defined by the Friedrich's Theorem via the closed
semi-bounded quadratic form,
\begin{equation}\label{QF-Gen}
 \mathcal{Q}^{\alpha,\gamma}_{h,\Omega}(u):=\norm{(-ih\nabla+A)u}^{2}_{L^{2}(\Omega)}+h^{1+\alpha}\Int{\partial\Omega}{}\gamma(x)|u(x)|^{2}dx\,.
\end{equation}
The assumption  $\gamma\in L^3(\partial\Omega)$ ensures that the
quadratic form in \eqref{QF-Gen} is semi-bounded. Since this does
not follow in a straightforward manner, we will recall the main
points of the classical proof in the appendix.

As is revealed from \eqref{eq:bc} and \eqref{QF-Gen}, the role of
the parameter $\alpha$ is to control the strength of the boundary
condition. Formally, we shall deal with the boundary term in
\eqref{QF-Gen} as a surface electric potential. This analogy is
already observed in \cite{FL}.

The quantity
\begin{equation}\label{eq:infb}
b=\inf_{x\in\overline\Omega}B(x)\,,
\end{equation}
is critical in the analysis of the spectrum of the operator
$\mathcal{P}^{\alpha,\gamma}_{h,\Omega}$. If  the domain $\Omega$ is
an exterior domain, i.e. the complement of a bounded subset, then
the operator $\mathcal{P}^{\alpha,\gamma}_{h,\Omega}$ has an
essential spectrum. In this case, the spectrum below $bh$ is
discrete, see e.g. \cite{KP}. When the domain $\Omega$ is an
interior domain, i.e. bounded, then by Sobolev embedding, the
operator $\mathcal{P}^{\alpha,\gamma}_{h,\Omega}$ is with compact
resolvent and its spectrum is purely discrete. If
$\sigma(\mathcal{P}^{\alpha,\gamma}_{h,\Omega})\cap(-\infty,bh)\not=\emptyset$,
then,
$$\sigma(\mathcal{P}^{\alpha,\gamma}_{h,\Omega})\cap(-\infty,bh)=\{e_1(h),e_2(h),\cdots\}\,,$$
where the terms of the sequence $(e_j(h))$ are eigenvalues of the
operator $\mathcal{P}^{\alpha,\gamma}_{h,\Omega}$ listed in
increasing order and by counting the multiplicity.

Let $\lambda\leq bh$. According to the aforementioned discussion, we
can introduce the two quantities,
\begin{align}
&E(\lambda;h,\gamma,\alpha)=-{\rm tr}\Big(\mathcal{P}^{\alpha,\gamma}_{h,\Omega}-bh\Big)_{-}=\sum_j\Big(e_j(h)-\lambda h\Big)_-\,,\label{eq:en}\\
&N(\lambda;h,\gamma,\alpha)={\rm tr}\Big(\mathbf 1_{(-\infty,\lambda h)}\left(\mathcal{P}^{\alpha,\gamma}_{h,\Omega}\right)\Big)\,.\label{eq:nb}
\end{align}
Notice that $E(\lambda;h,\alpha)$ is the sum of the absolute value
of the {\it negative} eigenvalues of
$\mathcal{P}^{\alpha,\gamma}_{h,\Omega}-\lambda h$ counting
multiplicities while the number of these eigenvalues is
$N(\lambda;h,\gamma,\alpha)$. In physics, $E(\lambda;h,\alpha)$ can
be interpreted as the energy of non-interacting fermionic particles
in $\Omega$ at chemical potential $\lambda h$ \cite{FG}.

The Lieb-Thirring inequality will ensure that the sum
$E(\lambda;h,\gamma,\alpha)$ is finite for all $\lambda\leq b$. This
will be discussed further in Section~\ref{sec:lb}. Concerning the
number of eigenvalues, $N(\lambda;h,\gamma,\alpha)$ is finite for
all $\lambda<b$. Actually, this energy level is strictly lower than
the bottom of the essential spectrum. For exterior domains, we may
have that the eigenvalues accumulate near $bh$, i.e.
$N(\lambda;h,\gamma,\alpha)=\infty$. In fact, it is proved that this
is the case when the magnetic field $B(x)$ is constant, see
\cite{CKP}.

The behavior of the two quantities in \eqref{eq:en} and
\eqref{eq:nb} in the semiclassical regime, i.e. when the
semiclassical parameter $h$ goes to $0$, is studied for the Neumann
problem in \cite{Fr} and \cite{Fo-Ka}. The Neumann problem
corresponds to $\gamma$ being identically $0$  in \eqref{eq:bc}.
When the magnetic field $B(x)=b$ is constant, then the results in
\cite{Fr} and \cite{Fo-Ka} assert that, if $h\to0_+$, then,
\begin{align}
&N(\lambda;h,\gamma=0,\alpha)= h^{-1/2}\,c_1(\lambda)+h^{-1/2}\,o(1)\,,\label{eq:nbFr}\\
&E(\lambda;h,\gamma=0,\alpha)=h^{1/2}\,c_2(\lambda)+h^{1/2}\,o(1)\,.\label{eq:nbFK}
\end{align}
The formula in \eqref{eq:nbFr} is valid for all $\lambda<b$ while
that in \eqref{eq:nbFK} is valid for all $\lambda\leq b$. It is
pointed in \cite{Fo-Ka} that the formulas in \eqref{eq:nbFr} and
\eqref{eq:nbFK} are equivalent when $\lambda<b$.

The quantities $c_1(\lambda)$ and $c_2(\lambda)$ are defined by
explicit expressions involving  spectral quantities for a harmonic
oscillator on the semi-axis. In \cite{Ka4}, it is derived an
analogue of \eqref{eq:nbFr} valid for a general function $\gamma\in
C^\infty(\partial\Omega)$ and constant magnetic fields. The key
issue in \cite{Ka4} was the analysis of a modified harmonic
oscillator on the semi-axis and a standard approximation of the
function $\gamma$ by a constant. The smoothness of the function
$\gamma$ vindicates the approximation of $\gamma$ by a constant
value as long as the approximation is done in a small domain.

In this paper, we aim to obtain analogues of \eqref{eq:nbFr} and
\eqref{eq:nbFK} under the relaxed assumptions that the magnetic
field is variable and the function $\gamma$ is no more smooth but
simply in $L^3(\partial\Omega)$. (This is the assumption needed to
define the self-adjoint operator in \eqref{Shr-op-Gen}). Also, we
add to the results of \cite{Ka4} by establishing a formula for
$E(\lambda;h,\gamma,\alpha)$ valid in the extended range
$\lambda\in(-\infty,b]$.

The approach we follow is by carrying out a reduction to a thin
boundary layer. This is easy to do. After localization in the thin
boundary layer, we localize in small sub-domains of the boundary
layer. In each small sub-domain, the operator is reduced to a one
defined with a constant magnetic field and a constant $\gamma$. The
reduced operator is defined in the half-plane. The reduction to a
constant magnetic field  is quiet standard as in \cite{Fr} and
\cite{Fo-Ka}. The non-trivial point is to reduce to a constant
$\gamma$ since the smoothness of $\gamma$ is dropped. We do this by
dealing with $\gamma$ as being a {\it surface} electric potential.
With this point of view, we borrow the methods in \cite{LSY} that
allow to approximate a non-smooth electrical potential by a smooth
one, and then one passes from the smooth potential to the constant
potential in the standard manner. Many errors will arise here. These
are controlled by various Lieb-Thirring inequalities, notably the
ones in \cite{Er-So, Sob, LW} and a remarkable inequality obtained
in \cite{Fo-Ka} valid in the torus.

We proceed in the statement of the main result of this paper. We
will need some notation regarding a harmonic oscillator in the
semi-axis. For $(\gamma,\xi)\in \R^{2}$, we denote by
\begin{equation}\label{Op-h-g-x}
  \mathfrak{h}[\gamma,\xi]= -\partial_{t}^{2}+(t-\xi)^{2}\quad {\rm in}\quad L^{2}(\R_{+}),
\end{equation}
the self-adjoint differential operator in $L^{2}(\R_+)$ associated
with the boundary condition $u^{\prime}(0)=\gamma u(0)$.
 The increasing sequence of eigenvalues of $\mathfrak{h}[\gamma,\xi]$ is
 $\{\mu_{j}(\gamma,\xi)\}_{j}$. By Sturm-Liouville theory, these
 eigenvalues are known to be simple and smooth functions of $\gamma$
 and $\xi$. These facts will be recalled precisely in a
 separate section.

In the following, $(x)_{-}=\max(-x,0)$ and $(x)_{+}=\max(x,0)$ denote the negative, respectively positive, part of a number $x\in\R$.

Our main result is

\begin{theorem}\label{thm:KN}
Suppose that the magnetic field satisfies,
$$b=\inf_{x\in\overline\Omega} B(x)>0\,.$$
Let $\lambda\leq b$, $\alpha\geq 1/2$ and $\gamma\in
L^3(\partial\Omega)$.  There holds:
\begin{itemize}
\item If $\alpha>1/2$,  then,
\[
\lim_{h\to0_+} \Big(h^{-1/2}\,E(\lambda;h,\gamma,\alpha)\Big)=\frac{1}{2\pi}\int_{\partial\Omega}\int_{\R}B(x)^{3/2}\Big(\mu_{1}(0,\xi)-\frac{\lambda}{B(x)}\Big)_{-}d\xi ds(x)\,.
\]
\item
If $\alpha=1/2$ and $\gamma\in L^\infty(\partial\Omega)$, then,
\begin{align*}
&\lim_{h\to0_+}
\Big(h^{-1/2}\,E(\lambda;h,\gamma,\alpha)\Big)\\
&\qquad\qquad=\frac{1}{2\pi}\sum_{p=1}^{\infty}\int_{\partial\Omega}\int_{\R}B(x)^{3/2}\Big(\mu_{p}\left(B(x)^{-1/2}\gamma(x),
\xi\right)-\frac{\lambda}{B(x)}\Big)_{-}d\xi ds(x).
\end{align*}
\end{itemize}
Here $ds(x)$ denotes integration with respect to arc-length along
the boundary $\partial\Omega$, and $E(\lambda;h,\gamma,\alpha)$ is
introduced in \eqref{eq:en}.
\end{theorem}

The results in Theorem~\ref{thm:KN} display the strength of the
boundary condition in \eqref{eq:bc}. We observe that the influence
of the Robin condition is not strong when $\alpha>\frac12$, since
the leading behavior of $E(\lambda;h,\gamma,\alpha)$ is essentially
the same as that for the Neumann condition (i.e. $\gamma=0$).

The sum
\begin{equation}\label{eq:sum}
\sum_{p=1}^{\infty}\int_{\R}\Big(\mu_{p}\left(\gamma,
\xi\right)-1\Big)_{-}d\xi\end{equation} is actually a sum of a
finite number of terms (for every fixed $\gamma$). The expression in
\eqref{eq:sum} is a continuous function of $\gamma$. This will be
proved in a separate section of this paper. Thus, we observe that
the terms appearing in Theorem~\ref{thm:KN} are well defined.

Due to the implicit nature of the quantity in \eqref{eq:sum}, it
seems hard to prove that the functional
$$\mathcal F(\gamma)=\sum_{p=1}^{\infty}\int_{\partial\Omega}\int_{\R}B(x)^{3/2}\Big(\mu_{p}\left(B(x)^{-1/2}\gamma(x),
\xi\right)-\frac{\lambda}{B(x)}\Big)_{-}d\xi ds(x)$$ is continuous
in $L^1(\partial\Omega)$. If this continuity is true, then the
result in Theorem~\ref{thm:KN} continues to hold under the relaxed
assumption that $\alpha=\frac12$ and $\gamma\in
L^3(\partial\Omega)$. This will be clear in the proof we provide to
Theorem~\ref{thm:KN}.

The methods we use do not allow us to obtain versions of
Theorem~\ref{thm:KN} valid for $\alpha<\frac12$. In this specific
regime, the sign of the function $\gamma$ will play a significant
role, as one can observe the results for the first eigenvalue in
\cite{Ka1}. The results in \cite{Ka1} suggest that the localization
to the boundary is very strong when $\alpha<\frac12$ and $\gamma$ is
negative. When $\alpha<\frac12$ and $\gamma>0$, then the effect of
the boundary is weak, and the situation is closer to the Dirichlet
boundary condition, for which the methods in \cite{CFFH} are
relevant.

Differentiation of the formulas in Theorem~\ref{thm:KN} with respect
to $\lambda h$ yields a formula for the number of eigenvalues. See
\cite{Fo-Ka, N} for a precise statement of this technique. The
formulas for the number of eigenvalues are collected in:

\begin{corollary}\label{cor:KN}
Let $\lambda<b$. Under the assumptions of Theorem~\ref{thm:KN},
there holds:
\begin{itemize}
\item If $\alpha>1/2$, then
\begin{equation}\label{FA}
\lim_{h\to0}\Big(h\, N(\lambda;h,\gamma,\alpha)\Big)=\frac{1}{2\pi}\iint_{
\{(x,\xi)\in
 \partial\Omega \times\R~:~B(x)\mu_1(0,\xi)<\lambda\}} B(x)^{1/2}d\xi ds(x)\,.
 \end{equation}
\item If $\alpha=1/2$, then
\begin{equation}\label{SA}
\begin{aligned}
&\lim_{h\to0}\Big(h\,
N(\lambda;h,\gamma,\alpha)\Big)\\
&\qquad=\frac{1}{2\pi}\sum_{p=1}^\infty
\iint_{
\{(x,\xi)\in\partial\Omega\times\R~:~B(x)\mu_p\left(B(x)^{-1/2}\gamma(x),\xi\right)<\lambda\}}
B(x)^{1/2}d\xi ds(x)\,.
\end{aligned}
\end{equation}
\end{itemize}
Here, $N(\lambda;h,\gamma,\alpha)$ is the number of eigenvalues
below $\lambda h$, introduced in \eqref{eq:nb}.
\end{corollary}
The proof of Corollary~\ref{cor:KN} is sketched below in Section~\ref{Sec:7}. We mention that a formula for the number of eigenvalues below the energy value
$\lambda=1$ is not available yet, even for the case of Neumann
boundary condition, i.e. $\gamma=0$. For a matter of illustration,
we include the following simple result in the case of Neumann
boundary condition and a square domain.

\begin{thm}\label{thm:SQ}
Suppose that the domain $\Omega$ is a square, the magnetic field is
constant, $\curl A=b$,  and that $\gamma=0$ in \eqref{eq:bc}. As
$h\to 0_+$, there holds,
\begin{equation}\label{eq:nb-s}
\limsup_{h\to0_+}\Big(h\,N(bh)\Big)=\frac{b|\Omega|}{2\pi}\,.\end{equation}
Here,
$$N(bh)=N(1;h,\gamma=0,\alpha=1)$$
is as introduced in \eqref{eq:nb}.
\end{thm}

In \cite{KK}, it is proved that the formula for the energy in
Theorem~\ref{thm:KN} is still valid  when the domain $\Omega$ is a
square and $\gamma=0$. This indicates an interesting observation,
namely, the energy
$$\sum_{j}(e_j(h)-bh)_-$$
is localized near the boundary, while the leading order expression
of the number of the eigenvalues below $bh$ is determined by the
bulk. The proof we give to Theorems~\ref{thm:KN} and \ref{thm:SQ}
 suggests that the eigenvalues strictly below $bh$ are
associated with eigenfunctions concentrated near the boundary. A
mathematically rigorous explanation of this point is still missing
in the literature.  Helpful information might be obtained by
computing the second correction term in \eqref{eq:nb-s}, expected to
be a boundary term. Toward that end, the methods in \cite{CFFH} must
prove useful.

If one considers the Dirichelt realization of the operator
$P^D=(-ih\nabla+A)^2$, then the number $ N(bh)$ is equal to $0$. If
$bh$ is an eigenvalue of $P^D$, then the corresponding ground state
can be extended by $0$ to all of $\R^2$. The min-max principle will
yield that this constructed function is an eigenfunction of the
Landau Hamiltonian in $\R^2$ with constant magnetic field $bh$. This
violates the description of the eigenfunctions of the lowest
eigenspace of the Landau Hamiltonian with a constant magnetic field,
since this space can not have compactly supported functions.  That
way we see that the lowest eigenvalue of $P^D$ is strictly larger
than $bh$.

\begin{rem}
A key ingredient in the proof of Theorem~\ref{thm:SQ}  is to compare
with a model Schr\"odinger operator with (magnetic) periodic
conditions. The advantage of this model operator is that its first
eigenvalue is known together with its multiplicity.
\end{rem}

\begin{rem}
We list some interesting open problems in connection with
Theorem~\ref{thm:SQ}:
\begin{itemize}
\item Inspection of the asymptotics in Theorem~\ref{thm:SQ} for
general domains.
\item Inspecting if the result in Theorem~\ref{thm:SQ} is valid with
$\liminf$ replacing $\limsup$.
\item Inspection of the number $\mathsf n(bh)$ of eigenvalues of $P_{h,b,\Omega}$ in the interval
$(-\infty,bh)$. This question is related to the existence of a
non-zero function $u$ solving the problem:
$$P_{h,b,\Omega}u=bhu{\rm ~in~}\Omega\quad{\rm and}\quad
\nu\cdot(h\nabla-iA_{0})u=0{\rm ~on~}\partial\Omega\,.$$
\end{itemize}
\end{rem}

\section{Preliminaries}
\subsection{Variational principles}
In this section, we recall methods used in \cite{LSY} to establish
upper and lower bounds on the energy of eigenvalues.
\begin{lem}{\label{lem-VP-2}}
Let $\mathcal{H}$ be a semi-bounded self-adjoint operator on $L^{2}(\R^{3})$ satisfying
\begin{equation}\label{hypI}
\inf{\rm Spec}_{\rm ess}(\mathcal{H})\geq 0\,.
\end{equation}
Let $\{\nu_{j}\}_{j=1}^{\infty}$ be the sequence of negative
eigenvalues of $\mathcal H$ counting multiplicities. We have,
\begin{equation}\label{eq-var-2}
-\sum_{j=1}^{\infty}(\nu_{j})_{-}
 =\inf\Sum{j=1}{N}\big\langle \psi_{j},\mathcal{H}\psi_{j}\big\rangle,
\end{equation}
where the infimum is taken over all $N\in\mathbb{N}$ and orthonormal families $\{\psi_{1},\psi_{2},\cdots,\psi_{N}\}\subset D(H)$.
\end{lem}
The next lemma states another variational principle. It is used in several papers, e.g. \cite{LSY}.
\begin{lem}\label{lem-VP-3}
Let $\mathcal{H}$ be a self-adjoint semi-bounded operator satisfying the hypothesis \eqref{hypI}.
Suppose in addition that $(\mathcal{H})_{-}$ is trace class.
For any orthogonal projection $\gamma$ with range belonging to the domain of $\mathcal{H}$ and such that $\mathcal{H}\gamma$ is trace class, we have,
\begin{equation}\label{eq-var-2}
-\sum_{j=1}^{\infty}(\nu_{j})_{-} \leq {\rm tr}(\mathcal{H\gamma})\,.
\end{equation}
\end{lem}
\subsection{Existence of discrete spectrum of $\mathcal{P}_{h,\Omega}^{\alpha,\gamma}$}\label{Ess-spec}

If the domain $\Omega$ is bounded, it results from the compact
embedding of $\mathcal{D}(\mathcal{Q}^{\alpha,\gamma}_{h,\Omega})$
into $L^{2}(\Omega)$ that $\mathcal{P}_{h}$ has compact resolvent.
Hence the spectrum is purely discrete consisting of a sequence of
eigenvalues accumulating at infinity.

In the case of exterior domains, the operator ${\mathcal{P}}_{h}$
can have essential spectrum. In particular, we have the inequality
\begin{equation}\label{est-HM}
 \int_{\Omega}|(-ih\nb +{ A})u|^{2}dx\geq h\int_{\Omega}{B}(x)|u|^{2}dx, \quad \forall\,u\in C^{\infty}_{0}(\Omega).
\end{equation}

Using then a magnetic version of Persson's Lemma ( see \cite{Bo1,Per}), we get that
\[
\inf{\rm Spec}_{\rm ess}\mathcal{P}_{h}\geq h b.
\]
This is the reason behind considering the sum of eigenvalues that are below $bh$.
\subsection{Lifting with respect to the dimension}
Let $d\in \mathbb{N}$, and let
\[
A(x)=(a_{1}(x),a_2(x),\cdots,a_{d+1}(x))^{T},
\]
be a magnetic vector potential with real-values entries in $L^2_{\rm loc}(\R_{+}^{d+1})$.

 We introduce the operator $H_{d}(\gamma)$ defined via the quadratic form
\begin{equation}
h_{d}(\gamma)[u]= \iint_{\R^{d+1}_{+}}|(-i\nabla+A)u(x)|^2dx-\int_{\R^d}{\gamma}(x)|u(x)|^2dx.
\end{equation}
Here and in the sequel $\R^{d}_{+}=\R^{d-1}\times \R_{+}$.

We are going to show the following theorem following a strategy used
in \cite[Theorem~3.2]{LW} to generalize a Lieb-Thirring type
inequality to the case with magnetic field.
\begin{theorem}\label{Lb-thm}
Let $d\geq 1$, $A\in L^{2}_{\rm
loc}(\overline{\R^{d+1}_{+}},\R^{d+1})$ and $\gamma\in
L^{2\alpha+d}(\R^{d})$. Let $\alpha\geq 1/2$, then
\begin{equation}\label{LT-Gamma}
{\rm tr}[H(\gamma)]^{\alpha}_{-}\leq2 L_{\alpha,d}^{\rm cl} \int_{\R^{d}}\gamma_{+}^{2\alpha+d}dx,
\end{equation}
where $L_{\alpha,d}^{\rm cl}$ is defined by
\[
L_{\alpha,d}^{\rm cl}=\dfrac{\Gamma(\alpha+1)}{2^d \pi^{d/2}\Gamma(1+\alpha+d/2)}
\]
\end{theorem}
\begin{proof}
We shall prove~\ref{LT-Gamma} by induction over $d$. Notice that this operator is well-defined for $d=0$ and $\gamma$ a non-negative real number. In this case we have $H_{0}(\gamma)=(-i\partial_{y}u+a(y))^2$ and $u^{\prime}(0)=-\gamma u(0)$,  and one easily can find that this operator has one negative eigenvalue, namely $-\gamma_{+}^2$, associated with the eigenfunction $e^{-i \int_{0}^{y} a(\tau)d\tau}e^{-{\gamma_{+}}y}$. Hence
\[
{\rm tr}_{L^2(\R_{+})}[H_0(\gamma)]^{\alpha}_{-}=(\gamma_{+}^2)^{\alpha}
\]
which is the analogue of \eqref{LT-Gamma} for $d=0$.

Now fix $d\geq 1$ and suppose that the assertion is already proved for all smaller dimensions. We write $x=(x_1,x^{\prime})$ when $x_1\in\R$ and $x^{\prime}\in \R^{d-1}$ and note that
\[
H_{d}(\gamma)\geq (-i\partial_{x_1}+a_{1}(x))^2\otimes 1_{L^2(\R^{d}_{+})}-[H_{d-1}(\gamma(x_1,\cdot))]_{-}
\]
We now choose a gauge
 \[
 \phi(x)=\int_{0}^{x_1}a_{1}(\tau,x_2,\cdots,x_{d+1})d\tau.
 \]
and $\widetilde{u}(x)=e^{-i\phi}u(x)$ for all $u\in \mathcal{D}(H_{d}(\gamma))$. Then
\[
 \big\langle H_d(\gamma)u,u\rangle_{L^2(\R^{d+1}_{+})}\geq \int _{\R^{d+1}_{+}}|\partial_{x_1}\widetilde u|^2dx-\int_{\R}\big\langle e^{-i\phi}[H_{d-1}(\gamma(x_1,\cdot))]_{-}e^{i\phi}\widetilde u, \widetilde u\big\rangle_{L^2(\R^{d}_{+})}dx_{1}
\]
So by the variational principle
\begin{equation*}
{\rm tr}_{L^2(\R^{d+1}_{+})}[H_d(\gamma)]^\alpha_{-}\leq {\rm tr}_{L^2(\R)}\left[-\partial_{x_1}^2\otimes 1_{L^2(\R^{d}_{+})}-e^{-i\phi}[H_{d-1}(\gamma(x_1,\cdot))]_{-}e^{i\phi}\right]^{\alpha}_{-}
\end{equation*}
and the operator-valued Lieb-Thirring inequality \cite[corollary 3.5]{HLW}, it follows that
\begin{multline}
{\rm tr}_{L^2(\R)}\left[-\partial_{x_1}^2\otimes 1_{L^2(\R^{d}_{+})}-e^{-i\phi}[H_{d-1}(\gamma(x_1,\cdot))]_{-}e^{i\phi}\right]^{\alpha}_{-}\\
\leq 2L_{\alpha,1}^{\rm cl} \int_{\R}{\rm tr}_{L^2(\R^{d}_{+})}[H_{d-1}(\gamma(x_1,\cdot))]^{\alpha+1/2}_{-}dx_1.
\end{multline}
By induction hypothesis, the right hand side is bounded above by
\[
2L_{\alpha,1}^{\rm cl} L_{\alpha+1/2,d-1}^{\rm cl}\int_{\R}\int_{\R^{d-1}}\gamma_{+}^{d+2\alpha}dx^{\prime}dx_1=2 L_{\alpha,d}^{\rm cl}\int_{\R^{d}}\gamma_{+}^{d+2\alpha}dx,
\]
which establishes the assertion for dimension $d$ and completes the
proof of Theorem~\ref{Lb-thm}.
\end{proof}
\subsection{Rough energy bound for the cylinder} In this
section, we recall a remarkable inequality for  the Schr\"{o}dinger
operator
\begin{equation}
\mathcal{P}_{h,{b},S,T}=(-ih\nb+ {b}{\bf A}_{0})^{2}\qquad {\rm in}\qquad L^{2}\Big([0,S]\times(0,h^{1/2}T)\Big)\,.
\end{equation}
Here ${S}$, $T$ and $b$ are positive parameters. The magnetic
potential ${\bf A}_{0}$ is
$${\bf A}_0(s,t)=(-t,0)\,.$$
 Functions  in the domain of the
operator $\mathcal{P}_{h,{b},S,T}$ satisfy the periodic conditions
\[
u(0,\cdot)= u(S,\cdot)\qquad {\rm on}\quad (0,h^{1/2}T),
\]
Neumann condition at $t=0$,
 $$\partial_{t}u=0 \quad{\rm on }\quad t=0\,,$$
 and Dirichlet condition at $t=h^{1/2}T$.

In this particular case of a bounded domain, the operator has
compact resolvent and the spectrum consists of an increasing
sequence of eigenvalues $(e_{j})_{j\geq 1}$ tending to $+\infty$. We
define the energy of the sum of the eigenvalues as follows,
\begin{equation}\label{energy-Gen}
   \mathcal{E}(\ld,{b},S,T)=\Sum{j}{}\big(h{b}(1+\ld)-e_{j}\big)_{+}\,.
\end{equation}
In \cite{Fo-Ka}, the energy in \eqref{energy-Gen} is controlled by
the product $ST$. We recall this estimate in the next lemma.
\begin{lem}\label{energy}
There exist positive constants $T_{0}$ and $\ld_{0}$ such that, for
all $S>0$, $b>0$, $T\geq \sqrt{b}T_{0}$ and $\ld\in(0,\ld_{0})$, we
have,
\[
\mathcal{E}(\ld,b,S,T)\leq C(1+\ld)hb\left(\dfrac{ST}{\pi h}+1\right).
\]
\end{lem}
\subsection{Boundary coordinates}\label{Sec:BC}
The aim of this section is to define a new system of coordinates
near the boundary which allows us to approximate the magnetic
potential locally near the boundary by a new one corresponding to a
constant magnetic field. These coordinates are used in \cite{HM}.
Let $\Omega$ be a smooth, simply connected domain in $\R^{2}$.
Suppose that the boundary $\partial\Omega$ is $C^{4}$-smooth. Let
furthermore,
\[
\mathbb{R}/(|\partial \Omega |\mathbb{Z})\ni s\mapsto M(s)\in\partial\Omega
\]
 be a parametrization of $\partial\Omega$. The unit tangent vector of $\partial\Omega$ at the point $M(s)$ of the boundary is given by
\[
T(s):= M^{\prime}(s).
\]
We define the scalar curvature $k(s)$ by the following identity
\[
T^{\prime}(s)=k(s)\nu(s),
\]
where $\nu(s)$ is the unit vector, normal to to the boundary, pointing outward at the point $M(s)$.
We choose the orientation of the parametrization $M$ to be counterclockwise, so
\[
\det(T(s),\nu(s))=1, \qquad \forall s\in \mathbb{R}/(|\partial \Omega |\mathbb{Z}) .
\]
For all $\delta>0$, we define
\[
\mathcal{V}_{\delta}= \{x\in\partial\Omega~:~{\rm dist} (x,\partial\Omega)<\delta\}\,.
\]
Let $t_0>0$. The map $\Phi=\Phi_{t_0}$ is defined as follows~:
\begin{equation}
\Phi:\mathbb{R}/(|\partial \Omega |\mathbb{Z})\times(0,t_{0})\mapsto x= M(s)-t\nu(s)\in \mathcal{V}_{t_{0}}.
\end{equation}
By smoothness of the boundary $\partial\Omega$, we may select $t_0$
sufficiently small so that $\Phi$ is invertible. Thus, for all
$x\in\mathcal{V}_{t_{0}}$, one can write
\begin{equation}\label{BC}
x\mapsto \Phi^{-1}(x):=(s(x),t(x))\in \mathbb{R}/(|\partial \Omega |\mathbb{Z})\times (0,t_{0}),
\end{equation}
where $t(x)={\rm dist}(x,\partial\Omega)$ and  $s(x)\in\mathbb{R}/(|\partial \Omega |\mathbb{Z})$ is associated with the point $M(s(x))\in\partial\Omega$ such that ${\rm dist}(x,\partial\Omega)= |x-M(s(x))|$.

The determinant of the Jacobian of the transformation $\Phi^{-1}$ is
\[
a(s,t)=1-tk(s).
\]
For all $u\in L^{2}(\mathcal V_{t_0})$, we define the function
\begin{equation}\label{tilde}
\widetilde u(s,t):= u(\Phi(s,t)).
\end{equation}
If $A=(A_{1},A_{2})$ is a vector field in $\mathcal{V}_{t_{0}}$, we
define the associated vector potential in the $(s,t)$-coordinates by
\begin{equation}
\begin{aligned}\label{mf-nc}
\widetilde A_{1}(s,t)&=(1-tk(s)) \vec{A}(\Phi(s,t))\cdot M^{\prime}(s),\\
\widetilde A_{2}(s,t)&= \vec{A}(\Phi(s,t))\cdot\nu(s).
\end{aligned}
\end{equation}
The new magnetic potential $\widetilde A$ satisfies,
\begin{equation}\label{Mf-dim2}
\Big[\dfrac{\partial \widetilde A_{2}}{\partial s}(s,t)- \dfrac{\partial \widetilde A_{1}}{\partial t}(s,t)\Big]ds\wedge dt= B(\Phi^{-1}(s,t))dx \wedge  dy= (1-tk(s))\widetilde B(s,t)ds\wedge dt.
\end{equation}
For all $u\in H^{1}_{A}(\mathcal{V}_{t_{0}})$, we have, with $\widetilde
u=u\circ \Phi$,
\begin{equation}
\int_{\mathcal{V}_{t_{0}}}|(-i\nabla +A)u|^{2}dx= \int_{0}^{|\partial\Omega|}\int_{0}^{t_0}\Big
[|(-i\partial_{s}+\widetilde A_{1})\widetilde u|^{2} +(1-tk(s))^{-2}
|(-i\partial_{t}+\widetilde A_{2})\widetilde
u|^{2}\Big](1-tk(s))dsdt\,, \end{equation}
and
\begin{equation}\label{unw}
\int_{\mathcal{V}_{t_{0}}}|u|^{2}dx= \int_{0}^{|\partial\Omega|}\int_{0}^{t_0}|\widetilde u(s,t)|^{2}(1-tk(s))dsdt.
\end{equation}

In the next proposition, it is constructed a gauge transformation
such that the magnetic potential in the new coordinates can be
approximated$-$up to a small error$-$by a new one corresponding to a
constant magnetic field. The proof is given in
\cite[Appendix~F]{FH-b}.
\begin{proposition}\label{prop:gauge}
Let $A \in C^{2}(\overline{\Omega},\R^{2})$. There exists a constant $C>0$ such that for all $S\in\big(0,|\partial\Omega|)$, $S_{0}\in[0,S]$ there exists a gauge function $\phi \in C^{2}(\big[0,S\big]\times[0,t_0])$ such that $\overline{A}:=\widetilde{A}(s,t)-\nabla_{(s,t)} \phi$, with $\widetilde{A}$ as defined in \eqref{mf-nc}, satisfies
\begin{equation}\label{Gg-dim2}
\overline{A}(s,t)=\begin{pmatrix}
\overline{A}_{1}(s,t)\\
\overline{A}_{2}(s,t)
\end{pmatrix} = \begin{pmatrix}
- B_{0}t+ \beta(s,t)\\
0\\
\end{pmatrix}, \qquad (s,t)\in[0,S]\times [0,t_{0}],
\end{equation}
where $ B_{0}:=\widetilde B (S_{0},0) $ and for any $0< T\leq t_{0}$, we have
\begin{equation}\label{bnd-beta}
\sup_{(s,t)\in[0,S]\times[0,T]}|\beta(s,t)|\leq C(S^{2}+T^{2}).
\end{equation}
\end{proposition}
We shall frequently make use of the following standard lemma, taken from~\cite[Lemma~3.5]{Fr}.
\begin{lem}\label{Lem-apqf}
There exists a constant $C>0$ and for all $S_{1}\in[0,|\partial\Omega|)$, $S_2\in (S_1,|\partial\Omega|)$, there exists a function $\phi\in C^2([S_1,S_2]\times[0,t_0];\R)$ such that, for all
\[
S_0\in[S_1,S_2],\quad \mathcal{T}\in(0,t_0),\quad \varepsilon\in [C\mathcal{T},Ct_0],
\]
and for all $u\in H^{1}_{A}(\Omega)$ satisfying
\[
{\rm supp}~\widetilde{u}\subset [S_1,S_2]\times[0,\mathcal{T}],
\]
one has the following estimate,
\begin{multline}
\left| \int_{\Omega}|(-ih\nb+A)u|^2dx-\int_{\R^2_{+}}|(-ih\nb+\widetilde{B}{\bf A}_{0})e^{i\phi/h}\widetilde u|^2 dsdt   \right|\\
\leq \int_{\R^2_{+}}\Big(  \varepsilon  |(-ih\nb+\widetilde{B}{\bf A}_{0})e^{i\phi/h}\widetilde u|^2+C\varepsilon^{-1} (S^2+\mathcal{T}^2)^2|\widetilde u|^2                        \Big)dsdt.
\end{multline}
Here, $\R^2_{+}=\R\times\R_{+}$, $S=S_2-S_1$, $\widetilde{B}=\widetilde{B}(S_0,0)$, the function $\widetilde{u}$ is associated to $u$ by $0$ on $\R^2_{+}\setminus {\rm supp}~\widetilde{u}$.
\end{lem}
\section{A family of one-dimensional differential operators}
We are concerned in this section with the analysis of a family of ordinary differential operators with Robin boundary condition. For $\xi\in\R$, we consider the operator $\mathfrak{h}[\gamma,\xi]$ in $L^{2}(\R_{+})$ associated with the operator $-\frac{d^{2}}{dt^{2}}+(t-\xi)^{2}$, i.e.
\begin{equation}
\mathfrak{h}[\gamma,\xi]:= -\dfrac{d^{2}}{dt^{2}}+(t-\xi)^{2},\qquad \mathcal{D}(\mathfrak{h}[\gamma,\xi])= \{u\in B^{2}(\R_{+})~:~u^{\prime}(0)=\gamma u(0)\}.
\end{equation}
Here, for a given $k\in\mathbb{N}$, the space $B^{k}(\R_{+})$ is defined as~:
\begin{equation}\label{spaceBk}
B^{k}(\R_{+})= \{u\in L^{2}(\R_{+})~:~ t^{p}u^{(q)}(t)\in L^{2}(\R_{+}),\quad \forall p,q \quad{\rm s.t.}\quad p+q\leq k\},
\end{equation}
where $u^{(q)}$ denote the distributional derivative of order $q$ of $u$.

The operator $\mathfrak{h}[\gamma,\xi]$ is associated with the closed quadratic form
\begin{equation}
B^{1}(\R_{+})\ni u\mapsto \mathfrak{q}[\gamma,\xi]:= \int_{0}^{\infty}(|u^{\prime}(t)|^{2}+|(t-\xi)u|^{2})dt,
\end{equation}
where $B^{1}(\R_{+})$ is defined in \eqref{spaceBk}.

It is easy to see that $\mathfrak{h}[\gamma,\xi]$ has compact resolvent since the embedding $B^{1}(\R_{+})$ into $L^{2}(\R_{+})$ is compact. Hence the spectrum of $\mathfrak{h}[\gamma,\xi]$ is purely discrete consisting of an increasing sequence of positive eigenvalues $\{\mu_{j}(\gamma,\xi)\}_{j=1}^{\infty}$.

The lowest eigenvalue of $\mathfrak{h}[\gamma,\xi]$ is defined via the min-max principle by~:
\[
\mu_{1}(\gamma,\xi)=\inf_{u\in B^{1}(\R_{+}),\, u\neq 0}\dfrac{\mathfrak{q}[\gamma,\xi](u)}{\norm{u}^{2}_{L^{2}(\R_{+})}}.
\]
It follows from standard Sturm-Liouville theory that all the
eigenvalues $\mu_j(\gamma,\xi)$ are simple, and $\mu_1(\gamma,\xi)$
has a positive ground state. Details are given in \cite{DaHe}.

We define the functions~:
\[
\Theta(\gamma):=\inf_{\xi\in\R}\mu_{1}(\gamma,\xi)\,,
\]
and
\[
\Theta_{j}(\gamma):=\inf_{\xi\in\R}\mu_{j}(\gamma,\xi) \quad(j\geq2).
\]
When $\gamma=0$, we shall write,
\begin{align}
\mathfrak{h}[\xi]:=\mathfrak{h}[0,\xi],\quad& \mu_{j}(\xi):=\mu_{j}(0,\xi),\quad\forall j\in\mathbb{N}\\
\Theta_{0}:=\Theta(0),\quad&\xi_{0}:=\xi(0).
\end{align}

The result in the next lemma is proved in \cite{Fr}.

\begin{lem}\label{mu2}
For all $\xi\in\R$, we have
$$\mu_{2}(\xi)>1\,.$$
\end{lem}

Next we collect results proved in \cite{Ka4}.

\begin{lem}\label{lem:Ka}
The following statements hold true.
\begin{enumerate}
\item  For all $\gamma\in\R$, we have,
$$\Theta_{2}(\gamma)>\Theta(\gamma).$$
\item For every $j\in\mathbb N$, the function $\xi\mapsto\mu_{j}(\gamma,\xi)$ is continuous
and satisfies
\begin{enumerate}
\item $\displaystyle\lim_{\xi\rightarrow-\infty}\mu_{j}(\gamma,\xi)=\infty$\,;
\item$\displaystyle\lim_{\xi\rightarrow\infty}\mu_{j}(\gamma,\xi)=2j+1$\,.
\end{enumerate}
\item Let $\gamma\in(-\infty,0)$ and $j\in\mathbb{N}$. Then
$\Theta_{j}(\gamma)<2j+1$ and for all
$b_{0}\in(\Theta_{j}(\gamma),2j+1)$, the equation
$\mu_{j}(\gamma,\xi)=b_{0}$ has exactly two solutions
$\xi_{j,-}(\gamma,b_{0})$ and $\xi_{j,+}(\gamma,b_{0})$.
Moreover,
\[
\{\xi\in\R~:~\mu_{j}(\gamma,\xi)<b_{0}\}=(\xi_{j,-}(\gamma,b_{0}),\xi_{j,+}(\gamma,b_{0}))\,.
\]
\item Let \[
U_{j}=\{(\gamma,b)\in\R^{2}~:~\Theta_{j}(\gamma)<b<2j+1\}\,.
\]
The functions
\[
U_{j}\ni(\gamma,b)\mapsto \xi_{j,\pm}(\gamma,b)
\]
admit continuous extensions
\[
\R\times(-\infty,2j+1)\mapsto \overline\xi_{j,\pm}(\gamma,b).
\]
\end{enumerate}
\end{lem}
For later use, we include
\begin{lem}\label{Lem:|u0|^2}
Let $\gamma\in\R$, and let $u_{j,\gamma}(\cdot;\xi)$ be the
normalized eigenfunction associated to the eigenvalue
$\mu_{j}(\gamma,\xi)$. It holds true that
\[
|u_{j,\gamma}(0;\xi)|^2\leq C(\mu_{j}(\gamma,\xi)+(\gamma^2+1)).
\]
\begin{proof}
Due to the density of $C^{\infty}_{0}(\overline{\R_{+}})$ in $H^{1}(\R_{+})$, we have for any function $u\in H^{1}(\R_{+})$,
\begin{equation}
|u(0)|^2=-2\int_{0}^{\infty}u^{\prime}(\eta)u(\eta)d\eta.
\end{equation}
The inequality of Cauchy-Schwarz gives us that, for any $\alpha>0$,
\begin{equation}\label{CSh}
|u(0)|^2\leq 2\|u^{\prime}\|_{L^2(\R_{+})}\|u\|_{L^2(\R_{+})}\leq \alpha\|u^{\prime}\|_{L^2(\R_{+})}^2+\alpha^{-1}\|u\|_{L^2(\R_{+})}^2.
\end{equation}
Assume $\gamma<0$ and choose $\alpha=-1/(2\gamma)$, it follows that
\begin{equation}\label{Bdry-term}
\gamma|u(0)|^2\geq -\frac{1}{2}\|u^{\prime}\|^2-2\gamma^2\norm{u}^2.
\end{equation}
Notice that for $u:=u_{j,\gamma}(\cdot,\xi)$, we have
\[
\|u_{j,\gamma}^{\prime}\|^2+\norm{(t-\xi)u_{j,\gamma}}^2+\gamma|u_{j,\gamma}(0;\xi)|^2=\mu_{j}(\gamma,\xi).
\]
Using \eqref{Bdry-term} with $u:=u_{j,\gamma}(\cdot,\xi)$ and adding $\|u_{j,\gamma}^{\prime}\|^2+\norm{(t-\xi)u_{j,\gamma}}^2$ on both sides, we obtain
\begin{equation}\label{Eq-mu-uprime}
\mu_{j}(\gamma,\xi)\geq  \frac{1}{2}\|u_{j,\gamma}^{\prime}\|^2-2\gamma^2.
\end{equation}
Note also that the inequality in \eqref{Eq-mu-uprime} is evidently
true for $\gamma\geq0$.

We infer from \eqref{CSh} that,
\[
|u_{j,\gamma}(0;\xi)|^2\leq 2\|u^{\prime}_{j,\gamma}\|^2 +2.
\]
Now we use the inequality in \eqref{Eq-mu-uprime} and the assumption
that $u_{j,\gamma}$ is normalized  in $L^2$ to deduce
\[
|u_{j,\gamma}(0;\xi)|^2\leq 4\mu_{j}(\gamma,\xi)  +(8\gamma^2+2)\,.
\]
\end{proof}
\end{lem}

In the next lemma, using the analysis in \cite[Theorem~2.6.2]{Ka},
we establish uniform decay estimates on the eigenfunctions
$u_{j,\gamma}$.

\begin{lem}
Let $\epsilon\in(0,1)$ and $K>0$.  There exists a constant
$C_{\epsilon,K}>0$ such that, if $|\xi|\leq K$ and
$\mu_{j}(\gamma,\xi)\leq 1$, then,
\begin{equation}\label{Decay}
\norm{e^{\epsilon(t-\xi)^2/2}u_{j,\gamma}(\cdot,\xi)}_{H^{1}(\R_{+};(t-\xi)\geq C_{\epsilon,K})}\leq C_{\epsilon,K}(1+\gamma_{-}+\gamma_{-}^2).
\end{equation}
\end{lem}
\begin{proof}
Let $\Phi:\R_{+}\to\R$ be a Lipschitz function in $\R_{+}$ such that
$\Phi'$ is compactly supported and $e^\Phi u_{j,\gamma}\in
L^2(\R_+)$. For all $\phi\in \mathcal{D}(\mathfrak{h}[\gamma,\xi])$,
we have the following identity:
\[
\langle\mathfrak{h}[\gamma,\xi]\phi,e^{\Phi}\phi\rangle_{L^2(\R_{+})}=\norm{(e^\Phi\phi)^\prime}^2_{L^2(\R_{+})}+\norm{(t-\xi)e^\Phi\phi}^2_{L^2(\R_{+})}+\gamma|e^{\Phi(0)}\phi(0)|^2_{L^2(\R_{+})}-\norm{\Phi^{\prime} e^\Phi\phi}_{L^2(\R_{+})}^2.
\]
Substituting $\phi=u_{j,\gamma}(\cdot;\xi)$, we obtain
\begin{multline}\label{Eq:Dec1}
\norm{(e^\Phi u_{j,\gamma}(\cdot;\xi))^\prime}_{L^2(\R_{+})}^2+\norm{(t-\xi)e^\Phi u_{j,\gamma}(\cdot;\xi)}_{L^2(\R_{+})}^2+ \gamma|e^{\Phi(0)}u_{j,\gamma}(0;\xi) |^2\\
=\mu_{j}(\gamma,\xi)\norm{e^{\Phi}u_{j,\gamma}(\cdot;\xi)}_{L^2(\R_{+})}^2+\norm{\Phi^{\prime} e^\Phi u_{j,\gamma}(\cdot;\xi)}_{L^2(\R_{+})}^2.
\end{multline}
Using Lemma~\ref{Lem:|u0|^2} and that $\mu_{j}(\gamma,\xi)\leq 1$, we deduce that
\[
|u_{j,\gamma}(0;\xi)|^2\leq 8\sqrt{1+\gamma_{-}^2}\,.
\]
Let us observe that
\[
\gamma|u_{j,\gamma}(0;\xi)|^2\geq -8\gamma_{-}\sqrt{1+\gamma_{-}^2}\,.
\]
Inserting this into \eqref{Eq:Dec2}, and again using that $\mu_{j}(\gamma,\xi)\leq 1$, it follows that
\begin{multline}\label{Eq:Dec2}
\norm{(e^\Phi u_{j,\gamma}(\cdot;\xi))^\prime}_{L^2(\R_{+})}^2+\norm{(t-\xi)e^\Phi u_{j,\gamma}(\cdot;\xi)}_{L^2(\R_{+})}^2\\
\leq\norm{e^{\Phi}u_{j,\gamma}(\cdot;\xi)}_{L^2(\R_{+})}^2+\norm{\Phi^{\prime} e^\Phi u_{j,\gamma}(\cdot;\xi)}_{L^2(\R_{+})}^2+ 8\gamma_{-}\sqrt{1+\gamma_{-}^2}e^{2\Phi(0)}.
\end{multline}
Let $N\in\mathbb N$ be sufficiently large. We choose the function
$\Phi$ to be
\[
\Phi:=\Phi_N=\left\{
\begin{array}{ll}
\epsilon \dfrac{(t-\xi)^2}{2}&{\rm if ~}t-\xi< N\,,\\
\epsilon \dfrac{N^2}{2}&{\rm if~}t-\xi\geq N \,.\end{array}
\right.\]
Implementing \eqref{Eq:Dec2}, we find
\begin{equation}\label{Eq:Dec3}
\int_{\R_{+}}\Big[{(e^\Phi u_{j,\gamma}(\cdot;\xi))^\prime}^2+\big[(1-\epsilon^2)(t-\xi)^2-1   \big]|e^\Phi u_{j,\gamma}(\cdot;\xi)|^2 \Big]dt
\leq 8\gamma_{-}\sqrt{1+\gamma_{-}^2}e^{\epsilon K^2/2}.
\end{equation}
This gives
\begin{equation}\label{Eq:Dec4}
\int_{N\geq (t-\xi)\geq a_{\epsilon}}\Big[|{(e^\Phi u_{j,\gamma}(\cdot;\xi))^\prime}|^2+   |{e^\Phi u_{j,\gamma}(\cdot;\xi)}|^2\Big]dt
\leq 8\gamma_{-}\sqrt{1+\gamma_{-}^2}e^{\epsilon K^2/2}+e^{\epsilon a_{\epsilon}^2},
\end{equation}
with $a_{\epsilon}= \sqrt{\frac{2}{1-\epsilon^2}}\,$. Now choose
$C_{\epsilon,K}=\max \{a_{\epsilon},16e^{\epsilon
K^2/2},2e^{\epsilon a^2_{\epsilon}}\}$. That way, we can rewrite
\eqref{Eq:Dec4} as follows,
\begin{equation}\label{Eq:Dec4}
\int_{N\geq (t-\xi)\geq C_{\epsilon,K}}\Big[|{(e^\Phi u_{j,\gamma}(\cdot;\xi))^\prime}|^2+   |{e^\Phi u_{j,\gamma}(\cdot;\xi)}|^2\Big]dt
\leq C_{\epsilon,K}(1+\gamma_{-}+\gamma_{-}^2).
\end{equation}
The estimate in \eqref{Eq:Dec4} is true for all $N> C_{\epsilon,K}$.
Sending $N$ to $\infty$ and using monotone convergence, we get the
estimate in \eqref{Decay}.
\end{proof}
The rest of this section is devoted to an analysis of the term in
\eqref{eq:sum}.

\begin{lem}\label{lem-sum-j}
Let $M>0$. There exist constants $j_0\in\mathbb N$ and $C>0$ such
that, for all $\gamma\in(-M,M)$, we have,
\[
\sum_{j=2}^{\infty}\int_{\R}(\mu_{j}(\gamma,\xi)-1)_{-}d\xi
=\sum_{j=2}^{\infty}\int_{\R}(\mu_{j}(\gamma,\xi)-1)_{-}d\xi\leq C.
\]
\end{lem}
 \begin{proof}
 Let us observe that for all $j\geq2$,
 \[
 \{\xi\in\R~:~\mu_{j}(\gamma,\xi)\leq 1\}\subset\{\xi\in\R~:~\mu_{2}(\gamma,\xi)\leq 1\}\,,
 \]
and for all $\gamma\in (-M,M)$, using the monotonicity of
$\eta\mapsto\mu_{2}(\eta,\xi)$, we have,
   \[
 \{\xi\in\R~:~\mu_{j}(\gamma,\xi)\leq 1\}\subset\{\xi\in\R~:~\mu_{2}(-M,\xi)\leq 1\}.
 \]
According to Lemma~\ref{lem:Ka}, there exists a constant $\ell>0$
such that
 \[
 \{\xi\in\R~:~\mu_{2}(-M,\xi)\leq 1\}\subset[-\ell,\ell],\qquad \forall\gamma\in (-M,M).
 \]
 We introduce constants $(\xi_{j}(M))_{j\geq 2}\subset[-\ell, \ell]$ by
 \[
 \mu_{j}(-M,\xi_{j}(M))=\min_{\xi\in[-\ell,\ell]}\mu_{j}(-M,\xi).
 \]
 Arguing as in the proof of \cite[Lemma~2.5]{Ka4}, we get,
 \begin{equation}\label{mu-j-gamma}
 \lim_{j\rightarrow\infty}\mu_{j}(-M,\xi_{j}(M))=\infty\,.
 \end{equation}
 Consequently, we may find $j_{0}\geq 2$ depending solely on $M$ such that
 \[
 \mu_{j}(-M,\xi_{j}(M))>1, \qquad (j> j_{0}).
 \]
 It follows that, for all $j> j_{0}$, $\xi\in[-\ell,\ell]$ and
 $\gamma\in(-M,M)$,
 \[
 \mu_{j}(\gamma,\xi)\geq \mu_{j}(-M,\xi)\geq \mu_{j}(-M,\xi_{j}(M))>1.
  \]
The result of Lemma~\ref{lem-sum-j} now follows upon noticing that,
for all $\gamma>-M$ and $\xi\in\R$,
 \[
 \mu_{j}(\gamma,\xi)> \mu_{2}(-M,\xi)\,.
 \]
\end{proof}

Again, the proof of \cite[Lemma~2.5]{Ka4} allows us to obtain:

\begin{lem}\label{lem-lim-mu-j}
 For all $M>0$, there holds,
\[
\lim_{j\rightarrow\infty}\left(\inf_{\xi\in\R} \mu_{j}(-M,\xi)\right)=\infty.
\]
 \end{lem}
 \begin{proof}
It has been established in \cite{DaHe} that there exists  a sequence
 $(\xi_{j}(M))_{j\in\N}$ such that, for all $j$,
\[
\inf_{\xi\in\R}\mu_{j}(-M,\xi)=\mu_{j}(-M,\xi_{j}(M))\,.
\]
Let us show that
\begin{equation}\label{lim-mu-j-gen}
    \lim_{j\rightarrow\infty}\mu_{j}(-M,\xi_{j}(M))=\infty\,.
\end{equation}
Suppose that \eqref{lim-mu-j-gen} were false. Then we can find a
constant $\mathcal{M}$ and a subsequence $j_{n}$ such that
\begin{equation}\label{xi-j-M}
\inf_{\xi\in\R}\mu_{j_{n}}(-M,\xi)=\mu_{j}(-M,\xi_{j_{n}}(M))\leq \mathcal{M},\qquad\forall~j\in\N.
\end{equation}
If $(\xi_{j_{n}}(M))_{n}$ is unbounded, we may find a subsequence,
denoted again by $(\xi_{j_{n}}(M))_{n}$, such that
$$\displaystyle{\lim_{n\rightarrow\infty}}\xi_{j_{n}}(M)=\infty.$$
Fix $j_{0}\in\mathbb{N}$ and let us observe that for all $j_{n}\geq
j_{0}$,
\begin{equation}\label{inq-egn}
\mu_{j_{n}}(-M,\xi_{j_{n}}(M))\geq\mu_{j_{0}}(-M,\xi_{j_{n}}(M))\,.
\end{equation}
On account of Lemma~\ref{lem:Ka}, we know that $\displaystyle{\lim_{\xi\rightarrow\infty}}\mu_{j_{0}}(-M,\xi)=2j_{0}+1$. Therefore, passing to the limit $n\rightarrow\infty$ in \eqref{inq-egn}, we obtain
\[
\liminf_{n\rightarrow\infty}\mu_{j_{n}}(-M,\xi_{j_{n}}(M))\geq 2j_{0}+1.
\]
Letting $j_{0}\rightarrow\infty$, we conclude,
\[
\liminf_{n\rightarrow\infty}\mu_{j_{n}}(-M,\xi_{j_{n}}(M))=\infty,
\]
which contradicts \eqref{xi-j-M}.

Now, if $(\xi_{j_n}(M))_{n}$ is bounded, we follow the proof of
Lemma~2.5 in \cite{Ka4}  and establish that
$$\lim_{n\to\infty}\mu_{j_{n}}(-M,\xi_{j_{n}}(M))=\infty\,,$$
which contradicts \eqref{xi-j-M}.
\end{proof}

\begin{lem}\label{cont-int}
The function
\[
\mathcal{I}:\R\ni\gamma\mapsto\Sum{j=2}{\infty}\Int{\R}{}( \mu_{j}(\gamma,\xi)-1)_{-}d\xi\,,
\]
is locally uniformly continuous.
\end{lem}
\begin{proof}
Let $m>0$. It is sufficient to establish,
\begin{equation}
  \left(\sup_{|\gamma|\leq m} \left|\mathcal{I}(\gamma+\tau)-\mathcal{I}(\gamma)\right| \right)\rightarrow 0 \text{   as   } \tau\rightarrow 0.
\end{equation}
Let $\tau_{1}\in (0,1)$. By monotonicity, it follows that for all
$\tau\in[-\tau_{1},\tau_{1}]$ and $j\geq2$,
\[
\{\xi\in\R~:~\mu_{j}(\gamma+\tau,\xi)\leq 1\}\subset\{\xi\in\R~:~  \mu_{2}(-m-\tau_{1},\xi)\leq 1 \}\,.
\]
We may find a constant $M>0$ depending only on $m$ such that
\begin{equation}
\forall \tau\in[-\tau_{1},\tau_{1}],\qquad\forall ~j\geq 2\,,\qquad
\{\xi\in\R~:~\mu_{j}(-m-\tau_{1},\xi)\leq 1 \}\subset[-M,M]\,.
\end{equation}
Let $\xi_{j}(M)$ be as in the proof of Lemma~\ref{lem-sum-j}, i.e.
\[
\forall~\xi\in[-M, M],\qquad\forall~\tau\in[-\tau_{1},\tau_{1}],\qquad\mu_{j}(\gamma+\tau,\xi)\geq\mu_{j}(-m-\tau_{1},\xi_{j}(M))\,.
\]
We get as in Lemma~2.5 in \cite{Ka4} and Lemma~\ref{lem-lim-mu-j}~:
\[
\lim_{j\rightarrow\infty}\mu_{j}(-m-\tau_{1},\xi_{j}(M))=\infty\,.
\]
Hence, we may find $j_{0}\geq 2$ depending solely on $m$ such that,
for all $j\geq j_0$,
\[
\mu_{j}(-m-\tau_{1},\xi_{j}(M))> 1\,,
\]
and consequently, for all $|\tau|\leq \tau_{1}$, we have,
\[
\Sum{j=2}{\infty}\Int{\R}{}(\mu_{j}(\gamma+\tau,\xi)-1)_{-}d\xi=\Sum{j=2}{j_{0}}\Int{\R}{}(\mu_{j}(\gamma,\xi)-1)_{-}d\xi.
\]
Therefore, we deal with a sum of $j_{0}$ terms with $j_{0}$
independent from $\tau$ and $\gamma$. So given
$k\in\{2,\cdots,j_{0}\}$ and setting
$\mathcal{I}_{k}(\gamma)=\Int{\R}{}\big(\mu_{k}(\gamma,\xi)-1\big)_{-}d\xi$, it is sufficient to show that
\begin{equation}
  \lim_{\underset{|\tau|\leq |\tau_{1}|}{\tau\rightarrow 0}}  \left(\sup_{|\gamma|\leq m} \big|\mathcal{I}_{k}(\gamma+\tau)-\mathcal{I}_{k}(\gamma)\big| \right)=0\,.
\end{equation}
Since the function $\gamma\mapsto\mu_{k}(\gamma,\xi)$ is continuous, the above formula 
is simply an application of dominated convergence.
\end{proof}
The next theorem is taken from \cite[Theorem~2.4.8]{Ka}.
 \begin{thm}
 There exist constants $C>0$ and $\eta>0$ such that, for all $\gamma\in\R$ and $\xi\in(\eta,+\infty)$, we have~:
 \begin{equation}\label{int-mu-1-finite}
  \big|\mu_{1}(\gamma,\xi)-1\big|\leq C(1+|\gamma|)\xi\exp{(-\xi^{2})}.
 \end{equation}
 \end{thm}
Let us introduce the function
\begin{equation}\label{eq:fctsum}
\mathcal{J}:\R\ni\gamma\mapsto\Sum{j=1}{\infty}\Int{\R}{}(\mu_{j}(\gamma,\xi)-1)_{-}d\xi\,.
\end{equation}
\begin{lem}\label{lim-int}
Let $\{\gamma_{h}\}_{h}$ be a real-sequence such that \(
\displaystyle{\lim_{h\rightarrow0}}\gamma_{h}=\gamma\in\R \). There
holds,
\[
\lim_{h\rightarrow 0}\mathcal{J}(\gamma_{h})=\mathcal{J}(\gamma).
\]
\end{lem}
\begin{proof}
We write,
\begin{multline}\label{dec-sum}
  \left|\mathcal{J}(\gamma_{h})-\mathcal{J}(\gamma)\right|
  \leq\Big|\Int{\R}{}(\mu_{1}(\gamma_{h},\xi)-1)_{-}d\xi-\Int{\R}{}(\mu_{1}(\gamma,\xi)-1)_{-}d\xi\Big|
  +| \mathcal{I}(\gamma_{h})-\mathcal{I}(\gamma)|\,.
\end{multline}
We treat the first term on the right hand side of \eqref{dec-sum}
 using the inequality \eqref{int-mu-1-finite}. That way, for every
$\epsilon$, there exists $h_0>0$ such that for all $h\in(0,h_0]$,
\[
\left|\Int{\R}{}(\mu_{1}(\gamma_{h},\xi)-1)_{-}d\xi\right|\leq \Int{\R}{}|\mu_{1}(\gamma_{h},\xi))-1|d\xi\leq  \int_{\R}g(\xi)d\xi\,,
\]
with
\[
g(\xi)= C(1+|\gamma|+\epsilon)\xi e^{-\xi^{2}/2}\in L^{1}(\R).
\]
By continuity of the function $\gamma\mapsto \mu_{1}(\gamma,\xi)$
and dominated convergence, it follows that
\[
\Int{\R}{}(\mu_{1}(\gamma_{h},\xi)-1)_{-}d\xi\rightarrow \Int{\R}{}(\mu_{1}(\gamma,\xi)-1)_{-}d\xi
\]
as $h\rightarrow 0$.
The second term in \eqref{dec-sum} converges to $0$ by Lemma~\ref{cont-int}. 
\end{proof}

\section{Eigenprojectors}
Recall that $\R^2_{+}=\R\times\R_{+}$. Consider $h,b>0$ and the magnetic potential
\begin{equation}\label{A0}
\R^{2}_{+}\ni (s,t)\mapsto {\bf A}_{0}(s,t)=(-t,0).
\end{equation}
In this section, we construct projectors on the (generalized)
eigenfunctions of the operator
\begin{equation}\label{Op-hs}
\mathcal{P}^{\alpha,\gamma}_{h,b,\R^{2}_{+}}= (-ih\nb+b{\bf
A_{0}})^{2} \text{ in } L^{2}(\R^{2}_{+})\,,
\end{equation}
whose  domain is
\[
\mathcal{D}(\mathcal{P}^{\alpha,\gamma}_{h,b,\R^{2}_{+}})
= \big\{ u\in L^2(\R^2_{+})~:~ (-ih\nabla+b{\bf A_{0}})^j\in L^2(\R^{2}_{+}),\quad j=1,2, \quad  \partial_{t}u=\gamma u\quad {\rm on}\quad t=0\big\}\,.
\]
 Consider an orthonormal family  $(u_{j,\gamma}(\cdot;\xi))_{j=1}^{\infty}$ of real-valued eigenfunctions of the operator $\mathfrak{h}[\gamma,\xi]$
 introduced in \eqref{Op-h-g-x}, i.e.
 \begin{equation}\label{Egv-HO}
\left\{
\begin{array}{ll}
-u_{j,\gamma}^{\prime\prime}(t;\xi)+(t-\xi)^{2}u_{j,\gamma}(t;\xi)=\mu_{j}(\gamma,\xi)u_{j,\gamma}(t;\xi),\quad\text{ in    }\R_{+},&\\
u_{j,\gamma}^{\prime}(0;\xi)=\gamma u_{j,\gamma}(0;\xi),&\\
\Int{\R_{+}}{}u_{j,\gamma}(t;\xi)^{2}dt=1.&
\end{array}
\right.
\end{equation}
Let $u\in \mathcal{D}(\mathcal{P}^{\alpha,\gamma}_{1,1,\R^{2}_{+}})$. Performing  a Fourier transformation with respect to  $s$, we
observe the formal relation,
\begin{equation}
\mathcal{P}^{\alpha,\gamma}_{1,1,\R^{2}_{+}}u =(2\pi)^{-1} \mathcal{F}_{\xi\rightarrow s}^{-1}\big(-\partial_{t}^2+(t-\xi)^2\big)\mathcal{F}_{s\rightarrow \xi}u\,.
\end{equation}
By the spectral theorem, we have
\[
\mathfrak{h}[\gamma,\xi]= \sum_{j=1}^{\infty} \big\langle \cdot, u_{j,\gamma}(\cdot;\xi)\big\rangle_{L^2(\R_{+})} u_{j,\gamma}(\cdot;\xi)\,,
\]
and consequently,
\[
\mathcal{P}^{\alpha,\gamma}_{1,1,\R^{2}_{+}}u 
=(2\pi)^{-1}\sum_{j=1}^{\infty} \big\langle\mathcal{F}_{s\rightarrow \xi}u , u_{j,\gamma}(\cdot;\xi)\big\rangle_{L^2(\R_{+})} \mathcal{F}_{\xi\rightarrow s}^{-1} u_{j,\gamma}(\cdot;\xi)\,.
\]
That way, for every 
$u\in
\mathcal{D}(\mathcal{P}^{\alpha,\gamma}_{1,1,\R^{2}_{+}})$, we have,
\begin{equation}
\big\langle \mathcal{P}^{\alpha,\gamma}_{1,1,\R^{2}_{+}}u,u \big\rangle_{L^2(\R^2_{+})}=
(2\pi)^{-1} \int_{\R}\sum_{j=1}^{\infty}\Big | \big\langle\mathcal{F}_{s\rightarrow \xi}u , u_{j,\gamma}(\cdot;\xi)\big\rangle_{L^2(\R_{+})}\Big|^2 d\xi\,.
\end{equation}

For every $j\in\mathbb N$ and $\xi\in\R$, we introduce the
eigenprojector $\Pi_j(\gamma,\xi)$ defined by the corresponding
bilinear form,
\[
\big\langle\Pi_{j}(\gamma,\xi)u,v\big\rangle_{L^2(\R^2_{+})}= \big \langle\mathcal{F}_{s\rightarrow \xi}u , u_{j,\gamma}(\cdot;\xi)\big\rangle_{L^2(\R_{+})}  \big\langle u_{j,\gamma}(\cdot;\xi),\mathcal{F}_{s\rightarrow \xi}v\big \rangle_{L^2(\R_{+})}\,.
\]
Through explicit calculations, it is easy to prove:
\begin{lem}\label{Eq:spec-id}
Let $u,v\in L^2(\R^2_{+})$. We have
\begin{equation}\label{Eq1}
(2\pi)^{-1}\int_{\R}\sum_{j=1}^{\infty} \big\langle \Pi_{j}(\gamma,\xi)u,v \big\rangle_{L^2(\R^2_{+})} d\xi=\big\langle u,v\big\rangle_{L^2(\R^2_{+})}.
\end{equation}
If in addition $u \in
\mathcal{D}(\mathcal{P}^{\alpha,\gamma}_{1,1,\R^{2}_{+}})$, then,
\begin{equation}\label{Eq2}
\big\langle u, \mathcal{P}^{\alpha,\gamma}_{1,1,\R^{2}_{+}}u \big\rangle_{L^2(\R^2_{+})} =
(2\pi)^{-1}\sum_{j=1}^{\infty} \int_{\R} \mu_{j} (\gamma,\xi)\big \langle \Pi_{j}(\gamma,\xi)u,u  \big\rangle_{L^2(\R^2_{+})} d\xi\,.
\end{equation}
\end{lem}
Let us introduce the unitary operator,
$$
U_{h,b}=L^{2}(\R^{2}_{+})\ni\varphi\mapsto U_{h,b}\varphi\in L^{2}(\R^{2}_{+}),
$$
such that, for all $x=(x_{1},x_{2}) \in\R^{2}_{+}$,
$$
(U_{h,b}\varphi)(x)=\sqrt{b/h}\,\varphi(\sqrt{b/h}\,x).
$$
Furthermore, we introduce the family of projectors,
\begin{equation}\label{U-pro}
\Pi_{j}(h,b;\gamma,\xi)=U_{h,b}\Pi_{j}(\gamma_{h,b},\xi)U_{h,b}^{-1}\,,
\end{equation}
with
\begin{equation}\label{gamma-h-b}
\gamma_{h,b}=h^{\alpha-1/2}b^{-1/2}\gamma.
\end{equation}
It is easy to check that
\begin{equation}\label{App-res}
U_{h,b}^{-1}\mathcal{P}^{\alpha,\gamma}_{h,b,\R^{2}_{+}}U_{h,b}=hb\mathcal{P}^{\alpha,\gamma_{h,b}}_{1,1,\R^{2}_{+}}\,.
\end{equation}
That way, we infer  from Lemma~\ref{Eq:spec-id}:
\begin{lem}\label{Egpro}
Let $u,v\in L^2(\R^2_{+})$. We have
\begin{equation}\label{Eq1}
(2\pi)^{-1}\int_{\R}\sum_{j=1}^{\infty} \big\langle \Pi_{j}(h,b;\gamma,\xi)u,v \big\rangle_{L^2(\R^2_{+})} d\xi=\langle u,v\big \rangle_{L^2(\R^2_{+})}.
\end{equation}
If in addition $u \in
\mathcal{D}(\mathcal{P}^{\alpha,\gamma}_{h,b,\R^{2}_{+}})$, then,
\begin{equation}\label{Eq2}
\big\langle u, \mathcal{P}^{\alpha,\gamma}_{h,b,\R^{2}_{+}}u \big\rangle =(2\pi)^{-1}hb\sum_{j=1}^{\infty} \int_{\R} \mu_{j} (h^{\alpha-1/2}b^{-1/2}\gamma,\xi) \big\langle \Pi_{j}(h,b;\gamma,\xi)u,u  \big \rangle_{L^2(\R^2_{+})} d\xi\,.
\end{equation}
\end{lem}
\section{Lower bound}\label{sec:lb}
In this section, we determine a lower bound of the trace
$-E(\lambda;h,\gamma,\alpha)$ consistent with the asymptotics
displayed in Theorem~\ref{thm:KN}.
%

Arguing as in \cite[Sec.~5.1]{Fo-Ka}, it follows from the
Lieb-Thirring inequality that the trace
$-E(\lambda;h,\gamma,\alpha)$ is finite.
\subsection{Decomposition of the energy}
Consider a partition of unity of $\R$,
\begin{equation}
\chi_{1}^{2}+\chi_{2}^{2}=1,\qquad {\rm supp}~\chi_{1}\subset (-\infty,1),\qquad {\rm supp}~\chi_{2}\subset \Big[\frac{1}{2},\infty\Big).
\end{equation}
We set for $k=1,2$, $x\in\R^{2}$,
\begin{equation}\label{Def-zetak}
\zeta_{k}(x)=\chi_{k}(t(x)),\qquad t(x)=
\begin{cases}
{\rm dist}(x,\partial\Omega)&{\rm if}\quad x\in\Omega\\
-{\rm dist}(x,\partial\Omega)&\text{  otherwise.}
\end{cases}
\end{equation}
Let $\delta:=\delta(h)\in(0,1)$ be a small parameter to be chosen
later. For $k=1,2$, we put,
\begin{equation}\label{Def-zetakh}
\zeta_{k,h}(x)=\zeta_{k}\Big(\dfrac{t(x)}{\delta(h)}\Big),\quad (x\in\overline{\Omega})\,,
\end{equation}
where $\zeta_{k}$ is introduced in \eqref{Def-zetak}.

Let $\{g_{j}\}_{j}$ be any orthonormal system in
$\mathcal{D}(\mathcal{P}^{\alpha,\gamma}_{h,\Omega})$. We aim to
prove a uniform lower bound of the following quantity,
\[
\sum_{j=1}^{N}(\mathcal{Q}_{h,\Omega}^{\alpha,\gamma}(g_{j})-\lambda h ).
\]
Thanks to the variational principle in Lemma~\ref{lem-VP-2}, this
will give us a lower bound of the trace
$-E(\lambda;h,\lambda,\alpha)$.

The IMS localisation formula yields
\begin{equation}\label{IMS-2}
\sum_{j=1}^{N}(\mathcal{Q}_{h,\Omega}^{\alpha,\gamma}(g_{j})-\lambda h )=\sum_{k=1}^{2}\Big( \mathcal{Q}_{h,\Omega}^{\alpha,\gamma}(\zeta_{k,h}g_{j})- \int_{\Omega}(\mathcal{V}_{h}+\lambda h)|{\zeta_{k,h}g_{j}}|^{2}dx \Big),\quad \mathcal{V}_{h}:= \sum_{k=1}^{2}|\nabla \zeta_{k,h}|^{2}.
\end{equation}
\subsection{The bulk term}
We will prove that the  bulk term in \eqref{IMS-2}  corresponding to
$k=2$ is an error term, i.e. of the order $o(h^{1/2})$. Thanks to
the variational principle in Lemma~\ref{lem-VP-3}, we have,
\begin{equation}\label{eq:blk}
\sum_{j=1}^{N}\Big( \mathcal{Q}_{h,\Omega}^{\alpha,\gamma}(\zeta_{2,h}g_{j})-
\int_{\Omega}(\mathcal{V}_{h}+\lambda h)|{\zeta_{2,h}g_{j}}|^{2} dx\Big)\geq
{\rm Tr}\Big(\big[\widetilde{\mathcal{P}}_{h}-(Bh+\mathcal{V}_{h})\big]{\bf 1}_{(-\infty,0)}\big(\widetilde{\mathcal{P}}^{\alpha,\gamma}_{h}-(Bh+\mathcal{V}_{h})\big)\Big),
\end{equation}
where
$\widetilde{\mathcal{P}}_{h}-(Bh+\mathcal{V}_{h})=(-ih\nabla+A)^2-(Bh+\mathcal{V}_{h})$
is the operator acting in $L^{2}(\R^{2})$. The trace on the right
side in \eqref{eq:blk} can be controlled using the Lieb-Thirring
inequality. The details are given in \cite[Sec.~5.2]{Fo-Ka}. That
way, we get,
\begin{equation}\label{eq:lb-blk}
\begin{aligned}
&\sum_{j=1}^{N}\Big( \mathcal{Q}_{h,\Omega}^{\alpha,\gamma}(\zeta_{2,h}g_{j})- \int_{\Omega}(\mathcal{V}_{h}+\lambda h)|{\zeta_{1,h}g_{j}}|^{2}dx\Big)\\
&\qquad\geq -C h^{2} \left( \int_{\R^{2}}\Big(\norm{h^{-1}B}_{L^{\infty}} (-h^{-2}\mathcal{V}_{h})_{-} + (-h^{-2}\mathcal{V}_{h})_{-}^{2}\Big)dx \right)\\
&\qquad\geq -C \Big(\dfrac{h}{\delta(h)}\big(1+\dfrac{h}{\delta(h)^{2}}\big)\Big).
\end{aligned}
\end{equation}
Therefore, we get,
\begin{equation}\label{Err-bulk}
\sum_{j=1}^{N}(\mathcal{Q}_{h,\Omega}^{\alpha,\gamma}(g_{j})-\lambda h )\geq \sum_{j=1}^{N}\Big( \mathcal{Q}_{h,\Omega}^{\alpha,\gamma}(\zeta_{1,h}g_{j})- \int_{\Omega}(\mathcal{V}_{h}+\lambda h)|{\zeta_{1,h}g_{j}}|^{2}\Big)- C\Big(\dfrac{h}{\delta(h)}\big(1+\dfrac{h}{\delta(h)^{2}}\big)\Big).
\end{equation}
Later on, we shall choose $\delta(h)$ in a manner that the first term (boundary term) on the right hand side above is the dominant term.
\subsection{The boundary term}
Here we handle the term corresponding to $k=1$ in \eqref{IMS-2}. By
assumption, $\partial\Omega$ has a finite number of connected
components. For simplicity of the presentation, we will perform the
computations in the case where $\partial\Omega$ has one connected
component. In the general case, we  work on each connected component
independently and then sum the resulting lower bounds.

Let us introduce a positive, smooth function $\psi\in L^{2}(\R)$,
supported in $(0,1)$ with the property that
\[
\int_{\R}\psi^{2}(s)ds=1.
\]
Recall the boundary coordinates $(s,t)$ introduced in \eqref{BC}. We
put
\begin{equation}\label{Def-psih}
\psi_{h}(x;\sigma)= \dfrac{1}{\delta(h)}\psi\left(  \dfrac{s(x)-\sigma}{\delta(h)}\right)\,,\quad(\sigma \in \R).
\end{equation}
Using again the IMS decomposition formula, we write,
\begin{multline}\label{eq:bnd}
\sum_{j=1}^{N}\Big(\mathcal{Q}_{h,\Omega}^{\alpha,\gamma}(\zeta_{1,h}g_{j})-\int_{\Omega}(\lambda
h +\mathcal{V}_{h}) |{\zeta_{1,h}g_{j}}|^{2}\Big) \\
= \int_{\R}\Big(
\mathcal{Q}_{h,\Omega}^{\alpha,\gamma}(\psi_{h} (x;\sigma)\zeta_{1,h}g_{j})-
(\lambda h +\mathcal{W}_{h}) |\psi_{h}(x;\sigma)\zeta_{1,h}g_{j}|^{2}
\Big)d\sigma,
\end{multline}
where
\begin{equation}
\mathcal{W}_{h}= \mathcal{V}_{h}+ h^{2}\int_{\R}|\nabla \psi_{h}(x,\sigma)|^{2}d\sigma.
\end{equation}
Let us denote by ($\Phi_{t_0}$ is the coordinate change \eqref{BC} valid near the boundary)
\begin{equation}\label{Asigma}
v_{j,h}(x;\sigma):= \psi_{h} (x;\sigma)\zeta_{1,h}(x)g_{j}(x), \qquad B_{\sigma}=B(\Phi(\sigma,0))\,,\quad A_{\sigma}(s,t)=B_{\sigma}{\bf A}_{0}(s,t)=(-B_{\sigma}t,0),
\end{equation}
where ${\bf A}_0$ is the magnetic potential introduced in \eqref{A0}.
From Lemma~\ref{Lem-apqf}, we infer that for all $\varepsilon\in(0,1)$,
\begin{multline}\label{f-est}
\int_{\Omega}\big|(-ih\nabla + A) v_{j,h}(x;\sigma)\big|^{2}dx\\ \geq
(1-\varepsilon) \int_{\R^{2}_{+}}\big|(-ih\nabla +A_{\sigma})\widetilde
v_{j,h,\sigma}|^{2}dsdt-C\varepsilon^{-1}\delta(h)^{4}
\int_{\R^{2}_{+}}|\widetilde v_{j,h,\sigma}|^{2}dsdt.
\end{multline}
Here, the function $\widetilde v_{j,h,\sigma}$ is defined by the
coordinate transformation as follows
\[
\widetilde v_{j,h,\sigma}(s,t)= e^{i\phi_\sigma(s,t)/h}\,\widetilde{v}_{j,h}(\Phi(s,t);\sigma)\,,
\]
where, for a function $u$, $\widetilde{u}$ is associated to $u$ by means of \eqref{tilde} and $\phi_{\sigma}$ is the phase factor from Lemma~\ref{prop:gauge}.

Combining the foregoing estimates yields
\begin{multline}\label{N-q-f'}
\int_{\Omega}|(-ih\nabla + A) v_{j,h}(x;\sigma)|^{2}dx- \int_{\Omega}(\lambda h + \mathcal{W}_{h})|v_{j,h}(x;\sigma)|^{2}dx\\
\geq  (1-\varepsilon) \int_{\R^{2}_{+}}|(-ih\nabla +A_{\sigma})\widetilde v_{j,h,\sigma}|^{2}dsdt
 - \Big(\lambda h(1+C\delta(h))+ C\varepsilon^{-1}\delta(h)^{4}) \Big) \norm{\widetilde v_{j,h,\sigma}}^{2}_{L^{2}(\R^{2}_{+})}  \\
 -   (1+C\delta(h))    \int \widetilde {\mathcal{W}}_{h}|\widetilde v_{j,h,\sigma}|^2dsdt\,.
\end{multline}
Consequently,
\begin{multline}\label{N-q-f}
\mathcal{Q}_{h,\Omega}^{\alpha,\gamma}(v_{j,h}(x;\sigma))-  \int_{\Omega}(\lambda h + \mathcal{W}_{h})|v_{j,h}(x;\sigma)|^{2}dx\\
\geq (1-\varepsilon)\int_{\R^{2}_{+}}|(-ih\nabla +A_{\sigma})\widetilde v_{j,h,\sigma}|^{2}dsdt
+h^{1+\alpha}\int_{\R}\gamma(s)|\widetilde v_{j,h,\sigma}(s,0)|^{2}ds\\
- \Big(\lambda h(1+C\delta(h))+ C\varepsilon^{-1}\delta(h)^{4}) \Big) \norm{\widetilde v_{j,h,\sigma}}^{2}_{L^{2}(\R^{2}_{+})}
-   (1+C\delta(h))  \int \widetilde {\mathcal{W}}_{h}|\widetilde v_{j,h,\sigma}|^2dsdt.
\end{multline}
The function $\gamma$ defined on $\partial\Omega$ can be viewed as a
function of the boundary variable $s\in(0,|\partial\Omega|)$. We
extend $\gamma$ by $0$ to a function in $L^3(\R)$.

Hereafter, we distinguish between the easy case when
$\alpha>\frac12$ and the harder case when $\alpha=\frac12$.

\subsection*{The regime $\alpha>\frac12$}
Let $\eta>0$. Thanks to \eqref{N-q-f}, we have the obvious
decomposition,
\begin{multline}\label{dec-eta0}
\sum_{j=1}^{N}\bigg\{\mathcal{Q}_{h,\Omega}^{\alpha,\gamma}(v_{j,h}(x;\sigma))-  \int_{\Omega}(\lambda h + \mathcal{W}_{h})|v_{j,h}(x;\sigma)|^{2}dx\bigg\}\\
\geq \sum_{j=1}^{N}\bigg[ (1-\eta)(1-\varepsilon)\int_{\R^{2}_{+}}|(-ih\nabla +A_{\sigma})\widetilde v_{j,h,\sigma}|^{2}dsdt+h^{1+\alpha}\int_{\R}\widetilde{\gamma}_{a,\sigma}|\widetilde v_{j,h,\sigma}(s,0)|^{2}ds
\\- \Big(\lambda h(1+C\delta(h))+ C\varepsilon^{-1}\delta(h)^{4}) \Big) \norm{\widetilde v_{j,h,\sigma}}^{2}_{L^{2}(\R^{2}_{+})}
 -   (1+C\delta(h))    \int \widetilde {\mathcal{W}}_{h}|\widetilde v_{j,h,\sigma}|^2dsdt  \bigg]\\+\eta(1-\varepsilon)
R_{h,\alpha,\eta,\sigma}(\widetilde
v_{j,h,\sigma})\,,
\end{multline}
where
\begin{multline}
R_{h,\alpha,\eta,\sigma}(\widetilde v_{j,h,\sigma})=
\sum_{j=1}^{N}\bigg[\int_{\R^{2}_{+}}|(-ih\nabla
+A_{\sigma})\widetilde v_{j,h,\sigma}|^{2}dsdt +
\eta^{-1}h^{1+\alpha}\int_{\R}
\dfrac{\gamma(s)}{1-\varepsilon}|\widetilde
v_{j,h,\sigma}(s,0)|^{2}ds\bigg].
\end{multline}
Furthermore, we define the operator $\widetilde\Gamma$ on $L^2([0,\delta(h)]\times(0,\delta(h)))$,
\[
\widetilde \Gamma f=\sum_{j=1}^{N}\langle   f, \widetilde v_{j,h,\sigma}\rangle_{L^{2}\big([0,\delta(h)]\times(0,\delta(h))\big)}  \widetilde v_{j,h,\sigma},
\]
which satisfies $0\leq \widetilde\Gamma\leq C\delta(h)^{-1}$ (in the sense of quadratic forms).

Denote by $\gamma_{h,b,\eta,\varepsilon}= \frac {h^{\alpha-1/2}B_{\sigma}^{-1/2}\gamma}{\eta(1-\varepsilon)}$. Thanks to the variational principle in Lemma~\ref{lem-VP-3}, we may write,
\begin{equation}\label{Eq-FErr0}
\begin{aligned} R_{h,\alpha,\eta
,\sigma}(\widetilde v_{j,h,\sigma})&= {\rm tr}\Big[ \mathcal{P}^{\alpha,\gamma/(\eta(1-\varepsilon))}_{h,B_{\sigma},\R^{2}_{+}}\widetilde\Gamma \Big]
 \geq -C\delta(h)^{-1}{\rm tr}\Big[  \mathcal{P}^{\alpha,\gamma/(\eta(1-\varepsilon))}_{h,B_{\sigma},\R^{2}_{+}}\Big]_{-}\\
 &\geq -C\delta(h)^{-1}hB_{\sigma}{\rm tr}\Big[  \mathcal{P}^{\alpha,\gamma_{h,b,\eta,\varepsilon}}_{1,1,\R^{2}_{+}}\Big]_{-}.
\end{aligned}
\end{equation}
Here the operator $ \mathcal{P}^{\alpha,\gamma_{h,b,\eta,\varepsilon}}_{1,1,\R^{2}_{+}}$ has been introduced in \eqref{Op-hs} and identified with the operator $H_{1}(-\gamma_{h,b,\eta,\varepsilon})$ defined in Lemma~\ref{Lb-thm}. Thus, it follows from Theorem~\ref{Lb-thm} (with $\alpha=1$) that

\begin{equation}\label{Eq-FErr0}
\begin{aligned} R_{h,\alpha,\eta,\sigma}(\widetilde v_{j,h,\sigma})
&\geq -C B_{\sigma}^{-1/2}h\delta^{-1}h^{3(\alpha-\frac12)}(1-\varepsilon)^{-3}\eta^{-3}\int_{\R}|\gamma(s)|^3ds \\
&\geq  -C B_{\sigma}^{-1/2}h\delta^{-1}h^{3(\alpha-\frac12)}(1-\varepsilon)^{-3}\eta^{-3}\|\gamma\|_3^3\,.
\end{aligned}
\end{equation}

Integrating \eqref{Eq-FErr0} with respect to
$\sigma\in(-\delta(h),|\partial\Omega|)$, we conclude that,
\begin{equation}\label{Error-energy0}
\begin{aligned}
\eta(1-\varepsilon) \int R_{h,\alpha,\eta,\sigma}(\widetilde v_{j,h,\sigma})d\sigma&\geq  -C B_{\sigma}^{-1/2}h\delta^{-1}h^{3(\alpha-\frac12)}
(1-\varepsilon)^{-2}\eta^{-2}\|\gamma\|_3^3\\
&=\mathcal{O}\big(h^{3(\alpha-\frac12)}\big)\,\eta^{-2}\delta^{-1}h\|\gamma\|_3^3.
\end{aligned}
\end{equation}
Selecting $\delta=h^{3/8}$, $\varepsilon=h^{1/4}$ and
$\eta=h^{1/32}$, we get that the error terms in \eqref{eq:lb-blk}
and \eqref{Error-energy0} are of the order $o(h^{1/2})$. Also, by
\cite[Proof of (5.26)]{Fo-Ka}, we have,
\begin{multline*}
\sum_{j=1}^{N}\bigg\{\mathcal{Q}_{h,\Omega}^{\alpha,\gamma}(v_{j,h}(x;\sigma))-
\int_{\Omega}(\lambda h +
\mathcal{W}_{h})|v_{j,h}(x;\sigma)|^{2}dx\bigg\}\\
\geq
-\frac{h^{1/2}}{2\pi}\int_{\partial\Omega}\int_{\R}B(x)^{3/2}\Big(\mu_{1}(0,\xi)-\frac{\lambda}{B(x)}\Big)_{-}d\xi
ds(x)-h^{1/2}o(1)\,.
\end{multline*}
Thus, we infer from \eqref{dec-eta0}, \eqref{eq:lb-blk} and
\eqref{IMS-2} that
\begin{equation}\label{lb-thm0}
-E(\lambda;h,\gamma,\alpha)\geq
-\frac{h^{1/2}}{2\pi}\int_{\partial\Omega}\int_{\R}B(x)^{3/2}\Big(\mu_{1}(0,\xi)-\frac{\lambda}{B(x)}\Big)_{-}d\xi
ds(x)-h^{1/2}o(1)\,.
\end{equation}

\subsection*{The regime $\alpha=\frac12$}
The calculations here are longer compared to the case
$\alpha>\frac12$. In the rest of this section, $\alpha=\frac12$. Let
$a>0$ and consider
\begin{equation}\label{gamma:a}
\gamma_{a}(s)=j_{a}\ast \gamma
\end{equation}
where
\[
j_{a}(s)= C_* a^{-1}j\Big(\frac{s}{a}\Big), \qquad j(s)=e^{-s^2}\,.
\]
Here $C_*$ is a normalization constant such that
$\displaystyle\int_{\R}j(s)\,ds=1$. By \cite[Theorem~2.16]{LL}, we
know that $\gamma_{a}\in C^{\infty}(\R)$ and, as $a\to0$,
\[
\gamma_{a}\rightarrow \gamma,\quad{\rm in~}L^3(\R)\,.
\]
By smoothness of $\gamma_a$, we have,
\begin{equation}\label{gamma-gamma:a}
|\gamma_{a}(s)-\gamma_{a}(\sigma)|\leq C a^{-2}|s-\sigma|\leq C a^{-2}\delta(h),
\end{equation}
valid on the support of the function $ v_{j,h,\sigma}$.

Also, we have the obvious decomposition,
\begin{equation}
\int_{\R}\gamma(s)|\widetilde v_{j,h,\sigma}(s,0)|^{2}ds
= \int_{\R}\gamma_{a}(s)|\widetilde v_{j,h,\sigma}(s,0)|^{2}ds+ \int_{\R}(\gamma(s)-\gamma_{a}(s))|\widetilde v_{j,h,\sigma}(s,0)|^{2}ds\,.
\end{equation}
Implementing the aforementioned estimates  in \eqref{N-q-f}, we
obtain,
\begin{multline}\label{N-q-f''}
\mathcal{Q}_{h,\Omega}^{\alpha,\gamma}(v_{j,h}(x;\sigma))-  \int_{\Omega}(\lambda h + \mathcal{W}_{h})|v_{j,h}(x;\sigma)|^{2}dx\\
\geq (1-\varepsilon)\int_{\R^{2}_{+}}|(-ih\nabla +A_{\sigma})\widetilde v_{j,h,\sigma}|^{2}dsdt
 +h^{3/2}\int_{\R}(\gamma_{a}(\sigma)-Ca^{-2}\delta(h))|\widetilde v_{j,h,\sigma}(s,0)|^{2}ds\\
+h^{3/2}\int_{\R}(\gamma(s)-\gamma_{a}(s))|\widetilde v_{j,h,\sigma}(s,0)|^{2}ds
- \Big(\lambda h(1+C\delta(h))+ C\varepsilon^{-1}\delta(h)^{4} \Big) \norm{\widetilde v_{j,h,\sigma}}^{2}_{L^{2}(\R^{2}_{+})}\\
-   (1+C\delta(h))  \int \widetilde {\mathcal{W}}_{h}|\widetilde v_{j,h,\sigma}|^2dsdt.
\end{multline}
Let $\eta>0$ and
\begin{equation}\label{eq:new-gamma}
{\gamma}_{a,\sigma}=\dfrac{\gamma_{a}(\sigma)-Ca^{-2}\delta(h)}{(1-\varepsilon)(1-\eta)}\,.
\end{equation}
We can rewrite \eqref{N-q-f''} in the alternative form,
\begin{multline}\label{dec-eta}
\sum_{j=1}^{N}\bigg\{\mathcal{Q}_{h,\Omega}^{\alpha,\gamma}(v_{j,h}(x;\sigma))-  \int_{\Omega}(\lambda h + \mathcal{W}_{h})|v_{j,h}(x;\sigma)|^{2}dx\bigg\}\\
\geq (1-\eta)(1-\varepsilon) \sum_{j=1}^{N}\bigg[\int_{\R^{2}_{+}}|(-ih\nabla +A_{\sigma})\widetilde v_{j,h,\sigma}|^{2}dsdt\\
+h^{3/2}\int_{\R}{\gamma}_{a,\sigma}|\widetilde
v_{j,h,\sigma}(s,0)|^{2}ds- {\lambda h}  \norm{\widetilde v_{j,\sigma}}^{2}_{L^{2}(\R^{2}_{+})}\bigg]
\\+\eta(1-\varepsilon)(1-\eta_0)
R^{(1)}_{h,\eta,\sigma}(\widetilde
v_{j,h,\sigma})+\eta\eta_0(1-\varepsilon)
R^{(2)}_{h,\eta,\sigma,a}(\widetilde v_{j,h,\sigma})\,,
\end{multline}
where $\eta_0\in(0,1/2)$,
\begin{align*}
R^{(1)}_{h,\eta,\sigma}(\widetilde
v_{j,h,\sigma})=&\sum_{j=1}^{N}\bigg[\int_{\R^{2}_{+}}|(-ih\nabla
+A_{\sigma})\widetilde v_{j,h,\sigma}|^{2}dsdt-\lambda h  \norm{\widetilde v_{j,\sigma}}^{2}_{L^{2}(\R^{2}_{+})}
\\
&+\left\{\lambda h\left(1-\frac1{1-\eta_0}\right)+\frac{\lambda h}{1-\eta_0}\left(1-\frac{1}{1-\varepsilon}\right)\right.\\
&\qquad\qquad\left.-\dfrac{2\eta^{-1}\lambda
hC\delta(h)+2C\eta^{-1}\varepsilon^{-1}\delta(h)^{4}}{(1-\varepsilon)(1-\eta_0)}\right\}
\norm{\widetilde v_{j,h,\sigma}}^{2}_{L^{2}(\R^{2}_{+})} \\
&-\dfrac{2\eta^{-1}(1+C\delta(h))}{(1-\varepsilon)(1-\eta_0)} \int \widetilde
{\mathcal{W}}_{h}|\widetilde v_{j,h,\sigma}|^2dsdt  \bigg],
\end{align*}
and
\begin{multline}
 R_{h,\eta,\sigma,a}^{(2)}(\widetilde v_{j,h,\sigma})=
\sum_{j=1}^{N}\bigg[\int_{\R^{2}_{+}}|(-ih\nabla
+A_{\sigma})\widetilde v_{j,h,\sigma}|^{2}dsdt +
2\eta_0^{-1}\eta^{-1}h^{3/2}\int_{\R}
\dfrac{\gamma(s)-\gamma_{a}(s)}{1-\varepsilon}|\widetilde
v_{j,h,\sigma}(s,0)|^{2}ds\bigg].
\end{multline}
The parameter $\eta_0$ will be selected sufficiently small but
fixed. Let us define the density matrix
\[
\widetilde \Gamma f=\sum_{j=1}^{N}\langle   f, \widetilde v_{j,h,\sigma}\rangle_{L^{2}\big([0,\delta(h)]\times(0,\delta(h))\big)}  \widetilde v_{j,h,\sigma},
\]
which satisfies $0\leq \widetilde\Gamma\leq C\delta(h)^{-1}$. Denote
by $\gamma_{\rm error}=2\eta_0^{-1}\eta^{-1}
\dfrac{\gamma(s)-\gamma_{a}(s)}{1-\varepsilon}$. Thanks to the
variational principle in Lemma~\ref{lem-VP-3} and the Lieb-Thirring
inequality in \eqref{Lb-thm}, we may write,
\begin{equation}\label{Eq-FErr}
\begin{aligned} R_{h,\eta,\sigma,a}^{(2)}(\widetilde v_{j,h,\sigma})&= {\rm tr}\Big[ \mathcal{P}^{\alpha,\gamma_{\rm error}}_{h,B_{\sigma},\R^{2}_{+}}\widetilde\Gamma \Big]
 \geq -C\delta(h)^{-1}{\rm tr}\Big[  \mathcal{P}^{\alpha,\gamma_{\rm error}}_{h,B_{\sigma},\R^{2}_{+}}\Big]_{-}\\
&\geq -C B_{\sigma}^{-1/2}h\delta^{-1}(1-\varepsilon)^{-3}\eta_0^{-3}\eta^{-3}\int_{\R}|\gamma(s)-\gamma_{a}(s)|^3ds \\
&\geq  -C B_{\sigma}^{-1/2}h\delta^{-1}(1-\varepsilon)^{-3}\eta_0^{-3}\eta^{-3}\|\gamma-\gamma_a\|_3^3\,.
\end{aligned}
\end{equation}
Let us make the following choice of the parameter $\delta$ and
$\varepsilon$,
\begin{equation}\label{eq:delta-eta}
\delta=\eta^{-3/4}h^{1/2},\qquad \varepsilon=h^{1/4}\,.
\end{equation}
Integrating \eqref{Eq-FErr} with respect to
$\sigma\in(-\delta(h),|\partial\Omega|)$, we conclude that,
\begin{equation}\label{Error-energy}
\begin{aligned}
\eta_0\eta(1-\varepsilon) \int R_{h,\eta,\sigma,a}^{(2)}(\widetilde v_{j,h,\sigma})d\sigma&\geq  -C B_{\sigma}^{-1/2}h\delta^{-1}
(1-\varepsilon)^{-2}\eta_0^{-2}\eta^{-2}\|\gamma-\gamma_a\|_3^3\\
&=\mathcal{O}\big(\,\eta_0^{-5/4}\eta^{-5/4}h^{1/2}\,\big)\|\gamma-\gamma_a\|_3^3.
\end{aligned}
\end{equation}

We estimate $R_{h,\eta,\sigma}^{(1)}(\widetilde v_{j,h,\sigma})$ using the
variational principle in Lemma~\ref{lem-VP-3} and the rough bound in
the cylinder in Lemma~\ref{energy}. Indeed, we have
\begin{multline*}
\Bigg\|\dfrac{2\eta^{-1}(1+C\delta(h))\widetilde{\mathcal{W}}_{h}}{(1-\varepsilon)(1-\eta_0)}-
\lambda
h\left\{\left(1-\frac1{1-\eta_0}\right)-\frac{1}{1-\eta_0}\left(1-\frac{1}{1-\varepsilon}\right)\right\}\\
+\dfrac{2C\eta^{-1}\lambda
h\delta(h)+C\varepsilon^{-1}\delta(h)^{4}}{(1-\varepsilon)(1-\eta_0)}
\Bigg\|_{L^{\infty}}  \leq \vartheta B_{\sigma}h\,,
\end{multline*}
where $\vartheta=\mathcal O(\eta)+\mathcal O(\eta_0)+o(1)$. We may
select $\eta$ and $\eta_0$ sufficiently small such that
$\vartheta<\lambda_0$, where $\lambda_0$ is the constant in
Lemma~\ref{energy}. That way, we may apply Lemma~\ref{energy}.
First, we write by the variational principle,
\begin{equation}
R_{h,\eta,\sigma}^{(1)}(\widetilde v_{j,h,\sigma}) \geq {\rm Tr}\Big[\Big(
\mathcal{P}_{h,B_{\sigma},\R^{2}_{+}}^{\alpha,0} -
B_{\sigma}h(1+\vartheta)\Big)\widetilde\Gamma \Big] \geq -C
\delta(h)^{-1}\mathcal{E}(\vartheta,B_{\sigma},
\delta(h),\delta(h))\,.
\end{equation}
Applying Lemma \ref{energy} and integrating with respect to $\sigma\in(-\delta(h), |\partial\Omega|)$, we arrive at
\begin{equation}\label{Error-energy*}
\eta(1-\varepsilon) \int R_{h,\eta,\sigma}^{(1)}(\widetilde v_{j,h,\sigma})d\sigma\geq -C \eta\delta(h)= -C\eta^{1/4}\, h^{1/2}.
\end{equation}
Collecting the estimates in \eqref{Error-energy},
\eqref{Error-energy*} and \eqref{dec-eta}, we get,
\begin{multline}\label{eq:lb-con}
\sum_{j=1}^{N}\bigg\{\mathcal{Q}_{h,\Omega}^{\alpha,\gamma}(v_{j,h}(x;\sigma))-  \int_{\Omega}(\lambda h + \mathcal{W}_{h})|v_{j,h}(x;\sigma)|^{2}dx\bigg\}\\
\geq (1-\eta)(1-\varepsilon)\sum_{j=1}^{N}\bigg[ \int_{\R^{2}_{+}}|(-ih\nabla +A_{\sigma})\widetilde v_{j,h,\sigma}|^{2}dsdt\\
+h^{3/2}\int_{\R}{\gamma}_{a,\sigma}|\widetilde
v_{j,h,\sigma}(s,0)|^{2}ds- \lambda h  \norm{\widetilde v_{j,h,\sigma}}^{2}_{L^{2}(\R^{2}_{+})}  \bigg]
-C\Big(\,\eta^{-5/4}\|\gamma-\gamma_a\|_3^3+\eta^{1/4}\,\Big)
h^{1/2}.
\end{multline}
The constant $C$ in the remainder term depends on the fixed
parameter $\eta_0$, but independent of the other parameters.  Notice
that the choice of $\delta$ and $\varepsilon$ in
\eqref{eq:delta-eta} makes the error in \eqref{eq:lb-blk} of the
order $\mathcal O(\sqrt{\eta}\,h^{1/2})$. Thus, collecting
\eqref{eq:lb-con}, \eqref{eq:lb-blk} and \eqref{IMS-2}, we get by
the variational principle in \eqref{lem-VP-2},
\begin{equation}\label{eq:lb-con'}
\begin{aligned}
-E(\lambda;h,\gamma,\alpha)\geq& (1-\eta)(1-\varepsilon)
\sum_{j=1}^{N}\int_{\R}\left[  \mathcal{Q}_{h,B_{\sigma},\R^{2}_{+}}^{\alpha ,{\gamma}_{a,\sigma}}(\widetilde v_{j,h,\sigma})- \lambda h  \norm{\widetilde v_{j,\sigma}}^{2}_{L^{2}(\R^{2}_{+})}  \right]d\sigma\\
&-C\Big(\,\eta^{-5/4}\|\gamma-\gamma_a\|_3^3+\eta^{1/4}\,\Big)
h^{1/2}.
\end{aligned}
\end{equation}
Here $\mathcal{Q}_{h,B_{\sigma},\R^{2}_{+}}^{\alpha,{\gamma}_{a,\sigma}}$ is the quadratic form associated to the operator in \eqref{Op-hs}
\subsection{The leading order term}
Here we continue to handle the case $\alpha=\frac12$. We are going
to estimate the leading term in \eqref{eq:lb-con'}, i.e.
\[
\sum_{j=1}^{N}\int_{\R}\left[  \mathcal{Q}_{h,B_{\sigma},\R^{2}_{+}}^{\alpha, {\gamma}_{a,\sigma}}(\widetilde v_{j,h,\sigma})- \lambda h  \norm{\widetilde v_{j,h,\sigma}}^{2}_{L^{2}(\R^{2}_{+})}  \right]d\sigma.
\]
Here $\gamma_{a,\sigma}$ is the constant introduced in
\eqref{eq:new-gamma}. Let
$$\widetilde\gamma_{h,\sigma}=h^{\alpha-1/2}B_\sigma^{-1/2}\gamma_{a,\sigma}=B_\sigma^{-1/2}\gamma_{a,\sigma}\,.$$
Recall the definition of the eigenprojector
$\Pi_p(h,B_\sigma;\widetilde\gamma_{h,\sigma};\xi)$ in
\eqref{U-pro}.
 By Lemma~\ref{Egpro}, we have,
\begin{equation}
  2\pi\Sum{j=1}{N}\mathcal{Q}^{\alpha,\gamma_{a,\sigma}} _{h,B_{\sigma},\R^{2}_{+}}(\widetilde v_{j,h,\sigma})\\
  =\Sum{j=1}{N}\Sum{p=1}{\infty}\int_{\R}\mu_{p}(\widetilde\gamma_{h,\sigma}\,,\,\xi)
  \Big \langle  \Pi_{p}(h,B_\sigma;\widetilde\gamma_{h,{\sigma}},\xi) \widetilde v_{j,h,\sigma},
  \widetilde v_{j,h,\sigma}\Big\rangle d\xi\,.
\end{equation}
Thus,
\begin{multline}\label{Eq-eig}
2\pi \sum_{j=1}^{N}\Big\{
\mathcal{Q}^{\alpha,\gamma_{a,\sigma}}
_{h,B_{\sigma},\R^{2}_{+}}(\widetilde v_{j,h,\sigma})
-\lambda h \int_{\R^{2}_{+}}|\widetilde v_{j,h,\sigma}|^{2}dsdt \Big\}\\
\geq
-hB_{\sigma}\sum_{p=1}^{\infty}\int_{\R}\Big(\mu_{p}(\widetilde\gamma_{h,\sigma}\,,\,\xi)-\frac{\lambda}{B_{\sigma}}\Big)_{-}
\sum_{j=1}^{N}\Big\langle
\Pi_{p}(h,B_\sigma;\widetilde\gamma_{h,{\sigma}},\xi) \widetilde
v_{j,h,\sigma}, \widetilde v_{j,h,\sigma}\Big\rangle d\xi.
\end{multline}
From the definition of
$\Pi_{p}(h,B_\sigma;\widetilde\gamma_{h,{\sigma}}\,,\,\xi)$ and the identity \eqref{unw}, it follows that
\begin{equation}\label{Eq-OS}
\begin{aligned}
\Big\langle  &\Pi_{p}(h,B_\sigma;\widetilde\gamma_{h,{\sigma}},\xi) \widetilde v_{j,h,\sigma}, \widetilde v_{j,h,\sigma}\Big\rangle_{L^2(\R^2_{+})}\\
&=\frac{B_{\sigma}}{h}
\Big|\big\langle \widetilde v_{j,h,\sigma},e^{-is\xi (h^{-1}B_{\sigma})^{1/2} }u_{p,\widetilde\gamma_{h,{\sigma}}}((h^{-1}B_{\sigma})^{1/2}t;\xi ) \big\rangle_{L^2(\R^2_{+})}\Big|^{2}\\
&\leq (1+C\delta(h))\frac{B_{\sigma}}{h}
\Big|\Big\langle g_{j}, \zeta_{1,h}\psi_{h}(x;\sigma)
U_{\Phi}^{-1}\Big(e^{-is\xi (h^{-1}B_{\sigma})^{1/2} }u_{p,\widetilde\gamma_{h,{\sigma}}}((h^{-1}B_{\sigma})^{1/2}t;\xi )\Big) \Big\rangle_{L^{2}(\Omega)}\Big|^{2}\,,
\end{aligned}
\end{equation}
where the transformation $U^{-1}_{\Phi}:\widetilde u\mapsto u$ is
associated with the coordinate transform $\Phi_{t_{0}}$ introduced
in \eqref{BC}. Next, since $\{g_{j}\}_{j=1}^{N}$ is an orthonormal
system in $L^{2}(\Omega)$, we have
\begin{multline}\label{Eq-OS_1}
\sum_{j=1}^{N}\Big|\Big\langle g_{j}, \zeta_{1,h}\psi_{h}(x;\sigma)  U_{\Phi}^{-1}\Big(e^{-is\xi (h^{-1}B_{\sigma})^{1/2} }u_{p,\widetilde\gamma_{h,{\sigma}}}((h^{-1}B_{\sigma})^{1/2}t;\xi )\Big) \Big\rangle_{L^{2}(\Omega)}\Big|^{2}\\
\leq \norm{\zeta_{1,h}\psi_{h}(x;\sigma)  U_{\Phi}^{-1}\Big(e^{-is\xi
(h^{-1}B_{\sigma})^{1/2}
}u_{p,\widetilde\gamma_{h,{\sigma}}}((h^{-1}B_{\sigma})^{1/2}t;\xi
)\Big)}^{2}_{L^{2}(\Omega)}\,.
\end{multline}
Putting \eqref{Eq-OS} and \eqref{Eq-OS_1} together, we get
\begin{equation}\label{eq:Pi}
\begin{aligned}
0\leq
&\sum_{j=1}^{N}\Big\langle  \Pi_{p}(h,B_\sigma;\widetilde\gamma_{h,{\sigma}},\xi)\widetilde v_{j,h,\sigma}, \widetilde v_{j,h,\sigma}\Big\rangle_{L^{2}(\R^{2}_{+})} \\
&\leq  (1+C\delta(h))
\frac{B_{\sigma}}{h}\norm{\zeta_{1,h}\psi_{h}(\sigma)
U_{\Phi}^{-1}\Big(e^{-is\xi (h^{-1}B_{\sigma})^{1/2} }u_{p,\widetilde\gamma_{h,{\sigma}}}((h^{-1}B_{\sigma})^{1/2}t;\xi )\Big)}^{2}_{L^{2}(\Omega)}\\
&\leq  (1+C\delta(h))\dfrac{B_{\sigma}}{h}
\Int{\R^{2}_{+}}{} (1-tk(s))
|\psi_{h}(s;\sigma)|^{2}|\zeta_{1,h}(t)|^{2}|u_{p,\widetilde\gamma_{h,{\sigma}}}(h^{-1/2} B_{\sigma}^{1/2}t; h^{-1/2} B_{\sigma}^{1/2},\xi)|^{2}dsdt\\
&\leq (1+C\delta(h))\dfrac{B_{\sigma}}{h}
\Int{\R}{} |\psi_{h}(s;\sigma)|^{2}ds \Int{\R_{+}}{}|u_{p,\widetilde\gamma_{h,{\sigma}}}(h^{-1/2} B_{\sigma}^{1/2}t;\xi)|^{2} dt\\
&= (1+C\delta(h))h^{-1/2} B_{\sigma}^{1/2}.
\end{aligned}
\end{equation}
Inserting this into \eqref{Eq-eig}, we find
\begin{multline}\label{Eq-eig-2}
2\pi \sum_{j=1}^{N}\Big\{ \mathcal{Q}^{\alpha,\gamma_{a,\sigma}} _{h,B_{\sigma},\R^{2}_{+}}(\widetilde v_{j,h,\sigma})
-\lambda h \int_{\R^{2}_{+}}|\widetilde v_{j,h,\sigma}|^{2}dsdt \Big\}\\
\geq
-h^{1/2}B_{\sigma}^{3/2}\sum_{p=1}^{\infty}\int_{\R}\Big(\mu_{p}(\widetilde\gamma_{h,{\sigma}}\,,\,\xi)-\frac{\lambda}{B_{\sigma}}\Big)_{-}d\xi-\mathcal
O(h^{1/2}\delta(h)).
\end{multline}
Fixing $a$ and $\eta$, we have,
${\gamma}_{h,\sigma}\rightarrow
\frac{\gamma_a(\sigma)}{1-\eta}$ as $h\rightarrow 0$.  It results
from Lemma~\ref{lim-int} that, if $h\to0$, then,
\begin{equation}
\sum_{p=1}^{\infty}\int_{\R}\Big(\mu_{p}(\widetilde\gamma_{h,{\sigma}}\,,\,\xi)-\frac{\lambda}{B_{\sigma}}\Big)_{-}d\xi\to
\sum_{p=1}^{\infty}\int_{\R}\Big(\mu_{p}\left(B_{\sigma}^{-1/2}\frac{\gamma_a(\sigma)}{1-\eta} ,\xi\right)
-\frac{\lambda}{B_{\sigma}}\Big)_{-}d\xi.
\end{equation}
Since the  function $\gamma_a$ is smooth and bounded (for every
fixed $a$), then by dominated convergence,
$$\int_{0}^{|\partial\Omega|}
\sum_{p=1}^{\infty}\int_{\R}\Big(\mu_{p}(\widetilde\gamma_{h,{\sigma}}\,,\,\xi)-\frac{\lambda}{B_{\sigma}}\Big)_{-}d\xi\,d\sigma
\to \int_{0}^{|\partial\Omega|}
\sum_{p=1}^{\infty}\int_{\R}\Big(\mu_{p}\left(B_{\sigma}^{-1/2}\frac{\gamma_a(\sigma)}{1-\eta} ,\xi\right)
-\frac{\lambda}{B_{\sigma}}\Big)_{-}d\xi d\sigma.$$
Inserting this and \eqref{Eq-eig-2} into \eqref{eq:lb-con'}, we get,
\begin{equation}\label{Eq-eig-3}
\begin{aligned}
\liminf_{h\to0}&\Big(-2\pi h^{-1/2}
E(\lambda;h,\gamma,\alpha)\Big)\\
&\geq
(1-\eta)\sum_{p=1}^{\infty}\int_{\partial\Omega}\int_{\R}B(x)^{3/2}\Big(\mu_{p}\left(B(x)^{-1/2}\frac{\gamma_a(x)}{1-\eta},\xi\right)-\frac{\lambda}{B(x)}\Big)_{-}d\xi
ds(x)\\
&\quad-C\Big(\,\eta^{-5/4}\|\gamma-\gamma_a\|_3^2+\eta^{1/4}\,\Big).
\end{aligned}
\end{equation}
Taking successively $\liminf_{a\to0_+}$ then
$\liminf_{\eta\to0_+}$, we arrive at,
\begin{multline}\label{Eq-eig-4}
\liminf_{h\to0}\Big(-2\pi h^{-1/2}
E(\lambda;h,\gamma,\alpha)\Big)\geq\\
\liminf_{\eta\to0_+}\left\{\liminf_{a\to0_+}
\sum_{p=1}^{\infty}\int_{\partial\Omega}\int_{\R}B(x)^{3/2}\Big(\mu_{p}\left(B(x)^{-1/2}\frac{\gamma_a(x)}{1-\eta},\xi\right)-\frac{\lambda}{B(x)}\Big)_{-}d\xi
ds(x)\right\}\,.
\end{multline}
If $\gamma\in L^\infty(\partial\Omega)$, then
$\|\gamma_a\|_\infty\leq \|\gamma\|_\infty$ and by dominated
convergence, the right side in \eqref{Eq-eig-4} is
$$\sum_{p=1}^{\infty}\int_{\partial\Omega}\int_{\R}B(x)^{3/2}\Big(\mu_{p}\left(B(x)^{-1/2}\gamma(x),\xi\right)-\frac{\lambda}{B(x)}\Big)_{-}d\xi
ds(x)\,.$$ Therefore, when $\gamma\in L^\infty(\partial\Omega)$ and $\alpha=\frac12$, we have the lower bound,
\begin{equation}\label{lb-thm2}
\liminf_{h\to0}\Big(-2\pi h^{-1/2}
E(\lambda;h,\gamma,\alpha)\Big)\geq\sum_{p=1}^{\infty}\int_{\partial\Omega}\int_{\R}B(x)^{3/2}\Big(\mu_{p}\left(B(x)^{-1/2}\gamma(x),\xi\right)-\frac{\lambda}{B(x)}\Big)_{-}d\xi
ds(x)\,.
\end{equation}

\section{Upper bound}

Let $\sigma\in[0,|\partial\Omega|)$, $\phi=\phi_{\sigma}$ be the
gauge from Proposition~\ref{prop:gauge}, $\zeta_{1,h}$ and
$\psi_{h}$ the functions from \eqref{Def-zetakh} and
\eqref{Def-psih} respectively. Let furthermore $\Phi=\Phi_{t_0}$ be
the coordinate transformation near the boundary given in \eqref{BC},
$B_{\sigma}=B(\Phi(\sigma,0))$ and $\breve{\gamma}_{h,\sigma}$  the
number introduced below in \eqref{def:gamma-tilde}.

Let $\xi\in \R$. If $\alpha=1/2$, we define the function
\begin{align*}
\widetilde{f}_{p,1/2}((s,t);h,\sigma,\xi)
:=B_{\sigma}^{1/4}h^{-1/4}
e^{-i\xi s\sqrt{B_{\sigma}/h}}
u_{p,\breve{\gamma}_{h,\sigma}}\big( B_{\sigma}^{1/2}h^{-1/2}t;\xi\big)e^{-i\phi_{\sigma}/h}\psi_{h}(s;\sigma)\zeta_{1,h}(t)\,,
\end{align*}
where $u_{p,\breve\gamma_{h,\sigma}}(\cdot;\xi)$ is the function from \eqref{Egv-HO}, and if $\alpha>1/2$, we define
\begin{align*}
\widetilde{f}_{p,\alpha}((s,t);h,\sigma,\xi):=B_{\sigma}^{1/4}h^{-1/4}e^{-i\xi s\sqrt{B_{\sigma}/h}}u_{p,0}\big( B_{\sigma}^{1/2}h^{-1/2}t;\xi\big)e^{-i\phi_{\sigma}/h}\psi_{h}(s;\sigma)\zeta_{1,h}(t)\,.
\end{align*}
Recall the coordinate transformation $\Phi$ valid near a
neighborhood of the point $x$ (see Subsection~\ref{Sec:BC}), and let
$x=\Phi^{-1}(y)$. We define $f_{p}(x;h,\sigma,\xi):={ \widetilde
f_{p}}((s,t);h,\sigma,\xi)$ by means of \eqref{tilde}. Let $K>0$. If
$\alpha=1/2$, we set,
\begin{equation}\label{eq:M}
M_{1/2}(h,\sigma,\xi,p,K)={{\bf 1}}_{\{(\sigma,\xi,p)\in[0,|\partial\Omega|)\mathbb\times\R\times\N~:~ \frac{\lambda}{B_{\sigma}}-\mu_{p}(\breve\gamma_{h,\sigma},\xi)\geq 0,~|\xi|\leq K\}},
\end{equation}
and if $\alpha>1/2$,
\begin{equation}\label{eq:M;alpha}
M_{\alpha}(h,\sigma,\xi,p,K)={{\bf 1}}_{\{(\sigma,\xi,p)\in[0,|\partial\Omega|)\mathbb\times\R\times\N~:~ p=1,\quad\frac{\lambda}{B_{\sigma}}-\mu_{1}(0,\xi)\geq 0,~|\xi|\leq K\}}\,.
\end{equation}
Since the calculations that we  perform will be done in the regimes
$\alpha=1/2$ and $\alpha>1/2$ independently, then, for the sake of
simplification, we will drop the subscript $\alpha$ in the
calculations below  and write $M$, $f_{p}$ instead of $M_{\alpha}$
and $f_{p,\alpha}$.

Let $f\in L^{2}(\Omega)$. We introduce
\begin{equation}\label{eq:Gamma}
(\Gamma f )(x)=(2\pi)^{-1}h^{-1/2}\iint B_{\sigma}^{1/2}\sum_{p=1}^{\infty}M(h,\sigma,\xi,p,K)\langle f_{p}(\cdot; h,\sigma,\xi), f \rangle f_{p}(x;h,\sigma,\xi)d\sigma d\xi\,.
\end{equation}
In Lemma~\ref{Pro-dm} below, we will prove that $\Gamma$ satisfies
the density matrix condition, namely,\break$0\leq \Gamma\leq
1+o(1)$. By the variational principle in Lemma~\ref{lem-VP-3}, an
upper bound of the sum of eigenvalues of
$\mathcal{P}_{h,\Omega}^{\alpha,\gamma}$ below $\lambda h$ follows
if we can prove an upper bound on
\begin{multline}\label{Eq:deftr}
{\rm tr}\big[(\mathcal{P}_{h,\Omega}^{\alpha,\gamma}-\lambda h){\Gamma}\big]\\
=(2\pi)^{-1}h^{-1/2}\iint\sum_{p=1}^{\infty}B_{\sigma}^{1/2}M(h,\sigma,\xi,p,K)\Big(\mathcal{Q}_{h,\Omega}^{\alpha,\gamma}(
f_{p}(x;h,\sigma,\xi))-\lambda h\norm{
f_{p}(x;h,\sigma,\xi)}^{2}\Big)d\sigma d\xi\,,
\end{multline}
where $\mathcal{Q}_{h,\Omega}^{\alpha,\gamma}$ is the quadratic form introduced in \eqref{QF-Gen}.
We will then estimate the quantity in \eqref{Eq:deftr} in the cases
$\alpha=1/2$ and $\alpha>1/2$ independently.
\subsection*{The regime $\alpha>1/2$}
In this subsection, we suppose that $\alpha>1/2$. We see in
\eqref{eq:M;alpha} that the definition of $M$ involves the first
eigenvalue $\mu_1(\cdot,\cdot)$ only. Consequently, the summation in
the definition of $\Gamma$ is restricted to the first term
corresponding to $p=1$. We observe that
\begin{equation}\label{Eq:QFalpha>1/2}
\mathcal{Q}_{h,\Omega}^{\alpha,\gamma}(f_{1}(x;h,\sigma,\xi))
=\mathcal{Q}_{h,\Omega}^{\alpha,0}(f_{1}(x;h,\sigma,\xi))+h^{1+\alpha}\int_{\R}\gamma(s)|\widetilde f_{1}((s,0);h,\sigma,\xi)|^2 ds.\\
\end{equation}
Easy computations lead to
\[
\int_{\R}\gamma(s)|\widetilde f_{1}((s,0);h,\sigma,\xi)|^2 ds\leq B_{\sigma}^{1/2}h^{-1/2}|u_{1,0}(0,\xi)|^{2}\int_{\R} (\gamma(s))_{+}|\psi_{h}(s;\sigma)|^{2}ds\,.
\]
Inserting this into \eqref{Eq:QFalpha>1/2}, we obtain
 \begin{multline}
\mathcal{Q}_{h,\Omega}^{\alpha,\gamma}(f_{1}(x;h,\sigma,\xi))
\leq \mathcal{Q}_{h,\Omega}^{\alpha,0}(f_{1}(x;h,\sigma,\xi))+B_{\sigma}^{1/2}h^{\alpha+1/2}|u_{1,0}(0,\xi)|^{2}\int_{\R} (\gamma(s))_{+}|\psi_{h}(s;\sigma)|^{2}ds.\\
\end{multline}
Now, we compute,
\begin{equation}\label{Eq:qf>1/2}
\begin{aligned}
    &{\rm tr}[(\mathcal{P}_{h,\Omega}^{\alpha,\gamma}-\lambda h)\Gamma]\\
    &\quad
    =  \iint (2\pi)^{-1}B_{\sigma}^{1/2}h^{-1/2}M(h,\sigma,\xi,p=1,K)\Big\{\mathcal{Q}_{h,\Omega}^{\alpha,\gamma}(f_{1}(x;h,\sigma,\xi))
-\lambda h\norm{f_{1}(x;h,\sigma,\xi)}^{2} \Big\}{d\sigma d\xi}\\
    &\quad\leq \Int{-K}{K}\Int{0}{|\partial\Omega|} (2\pi)^{-1}B_{\sigma}^{1/2}h^{-1/2}
    \Big\{\mathcal{Q}_{h,\Omega}^{\alpha,0}(f_{1}(x;h,\sigma,\xi)) -\lambda h \norm{f_{1}(x;h,\sigma,\xi)}^{2}\\
    &\qquad\qquad\qquad+     h^{\alpha+1/2}B_{\sigma}^{1/2}|u_{1,0}\big(0;\xi\big)|^2 \int_{\R}(\gamma(s))_{+}|\psi_{h}(s;\sigma)|^2ds  \Big\}{d\sigma d\xi}\\
    &\quad\leq
    \Int{-K}{K}\Int{0}{|\partial\Omega|} (2\pi)^{-1}B_{\sigma}^{1/2}h^{-1/2}\Big\{\mathcal{Q}_{h,\Omega}^{\alpha,0}(f_{1}(x;h,\sigma,\xi))
     -\lambda h \norm{f_{1}(x;h,\sigma,\xi)}^{2} \Big\}{d\sigma d\xi}\\
&\qquad\qquad\qquad    +    (2\pi)^{-1} h^{\alpha}\norm{B}_{L^{\infty}(\partial\Omega)}\int_{-K}^{K}|u_{1,0}\big(0;\xi\big)|^2 \int_{\R}\int_{0}^{|\partial\Omega|}(\gamma(s))_{+}|\psi_{h}(s;\sigma)|^2ds  {d\sigma d\xi}\\
\end{aligned}
\end{equation}
Using that
$\int_{0}^{|\partial\Omega|}\psi_{h}^{2}(s;\sigma)d\sigma=1$ and
taking into account the regularity of the function
$\xi\mapsto|u_{1,0}(0,\xi)|^2$, the second term on the right-hand
side of \eqref{Eq:qf>1/2} is estimated from above by
\[
(2\pi)^{-1} h^{\alpha}2K\norm{B}_{L^{\infty}(\partial\Omega)}\sup_{\xi\in[-K,K]}|u_{1,0}\big(0;\xi\big)|^2 \norm{\gamma}_{L^1(\partial\Omega)}\,,
\]
which is $o(h^{1/2})$ for fixed $K$. Also, by
\cite[Proof of (5.37)]{Fo-Ka}, the first term on the right-hand side of \eqref{Eq:qf>1/2} is bounded from above by,
\begin{equation*}
-\frac{h^{1/2}}{2\pi}\int_{\partial\Omega}\int_{-K}^{K}B(x)^{3/2}\Big(\mu_{1}(0,\xi)-\frac{\lambda}{B(x)}\Big)_{-}d\xi
ds(x)-h^{1/2}o(1)\,.
\end{equation*}
Thus, taking the successive limits $\limsup_{h\rightarrow 0^{+}}$
and $\lim_{K\rightarrow\infty}$, we obtain,
\begin{equation}\label{lb-thm0}
\limsup_{h\rightarrow 0}\Big(-h^{-1/2}E(\lambda;h,\gamma,1/2)\Big)\leq
-\frac{1}{2\pi}\int_{\partial\Omega}\int_{\R}B(x)^{3/2}\Big(\mu_{1}(0,\xi)-\frac{\lambda}{B(x)}\Big)_{-}d\xi
ds(x)\,,
\end{equation}
which gives the desired upper bound when $\alpha>1/2$.
\subsection*{The regime $\alpha=1/2$}
In this section, we restrict to the harder case $\alpha=1/2$. Here,
the definition of $M=M_{1/2}$ in \eqref{eq:M} involves the quantity,
\begin{equation}\label{def:gamma-tilde}
\breve{\gamma}_{h,\sigma}= \dfrac{B_{\sigma}^{-1/2}( \gamma_{a}(\sigma)+C a^{-2}\delta(h))}{1+\varepsilon}\,.
\end{equation}
In the definition of $\breve{\gamma}_{h,\sigma}$, $a\in(0,1)$ and
$\varepsilon\in(0,1)$ are fixed parameters, and  $\gamma_{a}$ is the
function introduced in \eqref{gamma:a}. Recall that, as $a\to0_+$,
$\gamma_a\to\gamma$ in $L^3(\partial\Omega)$.

We start by computing, for all $p\geq 1$,
\begin{equation}\label{norm:fj}
\begin{aligned}
\int_{\Omega}|f_{p}(x;h,\sigma,\xi)|^2dx
&=\int_0^{\infty}\int_0^{|\partial\Omega|}|\widetilde{f}_{p}((s,t);h,\sigma,\xi)|^2(1-tk(s))dsdt\\
&\quad\leq (1+\delta(h)\norm{k}_{\infty}) B_{\sigma}^{1/2}h^{-1/2}\\
&\qquad\times
\int_0^{\delta(h)}\int_0^{|\partial\Omega|}|u_{p,\breve\gamma_{h,\sigma}}(B_{\sigma}^{1/2}h^{-1/2}t;\xi)|^2|\zeta_{1,h}(t)|^2|\psi_{h}(s;\sigma)|^2dsdt\\
&\quad\leq (1+\delta(h)\norm{k}_{\infty}) B_{\sigma}^{1/2}h^{-1/2}
\int_0^{\delta(h)}|u_{p,\breve\gamma_{h,\sigma}}(B_{\sigma}^{1/2}h^{-1/2}t;\xi)|^2dt\\
&\quad= (1+\delta(h)\norm{k}_{\infty}),
\end{aligned}
\end{equation}
where we have used that the functions $\psi_{h}$ and
$u_{p,\breve\gamma_{h,\sigma}}$ are normalized in $s$ and $t$
respectively.
Again the normalization of $\psi_{h}$ implies that
\begin{equation}\label{norm:fj:0}
\begin{aligned}
\int_{\R}|\widetilde{f}_{p}((s,0);h,\sigma,\xi)|^2ds&\quad
= B_{\sigma}^{1/2}h^{-1/2}|u_{p,\breve\gamma_{h,\sigma}}(0;\xi)|^2|\zeta_{1,h}(0)|^2
\int_0^{|\partial\Omega|}|\psi_{h}(s,\sigma)|^2ds\\
&\quad\leq  B_{\sigma}^{1/2}h^{-1/2}|u_{p,\breve\gamma_{h,\sigma}}(0;\xi)|^2
\int_0^{|\partial\Omega|}|\psi_{h}(s;\sigma)|^2ds\\
&\quad= B_{\sigma}^{1/2}h^{-1/2}|u_{p,\breve\gamma_{h,\sigma}}(0;\xi)|^2.
\end{aligned}
\end{equation}
We also compute
\begin{equation}\label{norm-fk-lb-0}
\begin{aligned}
\int_{\Omega}|f_{p}(x;h,\sigma,\xi)|^2dx
&=\iint|\widetilde{f}_{p}((s,t);h,\sigma,\xi)|^2(1-tk(s))dsdt\\
&\geq (1-\delta(h)\norm{k}_{\infty}) B_{\sigma}^{1/2}h^{-1/2}\times\\
&\int_{0}^{\delta(h)}\int_{0}^{|\partial\Omega|}|u_{p,\breve{\gamma}_{h,\sigma}}(B_{\sigma}^{1/2}h^{-1/2}t;\xi)|^2|\zeta_{1,h}(t)|^2\left|\psi_{h}(s;\sigma)\right|^2dsdt.\\
&= (1-\delta(h)\norm{k}_{\infty})B_{\sigma}^{1/2}h^{-1/2}\int_{\R_{+}}|u_{p,\breve\gamma_{h,\sigma}}(B_{\sigma}^{1/2}h^{-1/2}t;\xi)|^2|\zeta_{1,h}(t)|^2dt.\\
\end{aligned}
\end{equation}
Let us write the last integral as
\begin{multline}
\int_{\R_{+}}|u_{p,\breve\gamma_{h,\sigma}}(B_{\sigma}^{1/2}h^{-1/2}t;\xi)|^2|\zeta_{1,h}(t)|^2dt
=\int_{\R_{+}}|u_{p,\breve\gamma_{h,\sigma}}(B_{\sigma}^{1/2}h^{-1/2}t;\xi)|^2dt\\
+\int_{\R_{+}}|u_{p,\breve\gamma_{h,\sigma}}(B_{\sigma}^{1/2}h^{-1/2}t;\xi)|^2(|\zeta_{1,h}(t)|^2-1)dt\\
=B_{\sigma}^{-1/2}h^{1/2}+ \int_{t\geq
\delta(h)/2}|u_{p,\breve\gamma_{h,\sigma}}(B_{\sigma}^{1/2}h^{-1/2}t;\xi)|^2(|\zeta_{1,h}(t)|^2-1)dt.
\end{multline}
Taking into account the support of $\zeta_{1,h}$, we can write,
\begin{multline}\label{Ag-est}
\int_{\R_{+}}|u_{p,\breve\gamma_{h,\sigma}}(B_{\sigma}^{1/2}h^{-1/2}t;\xi)|^2|\zeta_{1,h}(t)|^2dt\geq B_{\sigma}^{1/2}h^{-1/2}
- \int_{t\geq \delta(h)/2}|u_{p,\breve\gamma_{h,\sigma}}(B_{\sigma}^{1/2}h^{-1/2}t;\xi)|^2dt\\
= B_{\sigma}^{-1/2}h^{1/2}- \int_{t\geq
\delta(h)/2}e^{-\epsilon(B_{\sigma}^{1/2}h^{-1/2}t-\xi)^2/2}e^{\epsilon(B_{\sigma}^{1/2}h^{-1/2}t-\xi)^2/2}
|u_{p,\breve\gamma_{h,\sigma}}(B_{\sigma}^{1/2}h^{-1/2}t;\xi)|^2dt\,.
\end{multline}
 In observance of  the support of $M=M_{1/2}$,  we
see that $|\xi|\leq K$ and
\[
(B_{\sigma}^{1/2}h^{-1/2}t-\xi)^{2}\geq \big(b^{1/2}h^{-1/2}\frac{\delta(h)}{2}-\xi\big)^{2}\geq \frac{1}{8} bh^{-1}\delta(h)^{2}-2K^2.
\]
Implementing this into \eqref{Ag-est} and using the exponential
decay given in \eqref{Decay}, we find that
\begin{multline}\label{Ag-est-1}
\int_{\R_{+}}|u_{p,\breve\gamma_{h,\sigma}}(B_{\sigma}^{1/2}h^{-1/2}t;\xi)|^2|\zeta_{1,h}(t)|^2dt\\
\geq  B_{\sigma}^{1/2}h^{-1/2}\Big(1-C_{\epsilon,K}e^{-\epsilon(\frac{1}{8} bh^{-1}\delta(h)^{2}-2K^2)/2}(1+a^{-2}\delta(h)+(a^{-2}\delta(h))^2)\Big).
\end{multline}
In the last step we have used that $\gamma\in L^{\infty}$ together with the definition of $\breve\gamma_{h,\sigma}$ in \eqref{def:gamma-tilde}.

Inserting this into \eqref{norm-fk-lb}, we finally obtain
\begin{equation}\label{norm-fk-lb}
\int_{\Omega}|f_{p}(x;h,\sigma,\xi)|^2dx\geq
 (1-\delta(h)\norm{k}_{\infty})(1-C_{\epsilon,K} e^{-\epsilon(\frac{1}{8} bh^{-1}\delta(h)^{2}-2K^2)/2}(1+a^{-2}\delta(h)+(a^{-2}\delta(h))^2).
\end{equation}
Next we estimate the quadratic form. By Lemma~\ref{Lem-apqf},
we have for all $\varepsilon>0$,
\begin{equation}\label{eq:ub-e}
\begin{aligned}
&\mathcal{Q}_{h,\Omega}^{\alpha,\gamma}(f_{p}(x;h,\sigma,\xi))\\
&\quad
=\int_{\Omega}|(-ih\nabla+A)f_{p}(x;h,\sigma,\xi)|^2dx+h^{3/2}\int_{\partial\Omega}\gamma(x)|f_{p}(x;h,\sigma,\xi)|^2 dx\\
&\quad\leq (1+\varepsilon)\int_{\R^2_{+}}|(-ih\nabla+A_{\sigma})e^{i\phi_{\sigma}/h}\widetilde{f}_{p}((s,t);h,\sigma,\xi)|^2 dsdt\\
&\qquad+ C \varepsilon^{-1}\delta(h)^4\int_{\R^2_{+}}|\widetilde f_{p}((s,t);h,\sigma,\xi)|^2 dsdt +h^{3/2}\int_{\R}\gamma(s)|\widetilde f_{p}((s,0);h,\sigma,\xi)|^2 ds\,.\\
\end{aligned}
\end{equation}
Writing $\gamma=\gamma_{a}+(\gamma-\gamma_{a})$, it follows that
\begin{equation}\label{QF-1}
\begin{aligned}
&\mathcal{Q}_{h,\Omega}^{\alpha,\gamma}(f_{p}(x;h,\sigma,\xi))\\
&\quad \leq (1+\varepsilon)B_{\sigma}^{1/2}h^{-1/2}\int \psi_{h}^2(s,\sigma)
\big| (-ih \nabla+A_{\sigma})e^{-i\xi s\sqrt{B_{\sigma}/h}}u_{p,\breve\gamma_{h,\sigma}}\big( B_{\sigma}^{1/2}h^{-1/2}t;\xi\big)\big|^2 dsdt\\
&\qquad+\Big((1+\varepsilon)Ch^2\delta(h)^{-2}+C \varepsilon^{-1}\delta(h)^4\Big) \int_{\R^2_{+}}|\widetilde f_{p}((s,t);h,\sigma,\xi)|^2 dsdt\\
&\qquad+h^{3/2}\int_{\partial\Omega}\gamma_{a}(s)|\widetilde f_{p}((s,0);h,\sigma,\xi)|^2 ds\\
&\qquad+hB_{\sigma}^{1/2}|u_{p,\breve\gamma_{h,\sigma}}\big(0;\xi\big)|^2 |\zeta_{1,h}(0)|^2\int_{\R}(\gamma(s)-\gamma_{a}(s))|\psi_{h}(s;\sigma)|^2ds,
\end{aligned}
\end{equation}
where $A_{\sigma}$ is defined in \eqref{Asigma}.  Plugging
\eqref{norm:fj} and \eqref{norm:fj:0} into \eqref{QF-1}, and using
\eqref{gamma-gamma:a}, we find
\begin{equation}\label{Eqt:Qfj}
\begin{aligned}
&\mathcal{Q}_{h,\Omega}^{\alpha,\gamma}(f_{p}(x;h,\sigma,\xi))\\
&\quad \leq
(1+\varepsilon)\Big\{B_{\sigma}^{1/2}h^{-1/2}\int_{\R^{2}_{+}}
\Big| (-ih \nabla+A_{\sigma})e^{-i\xi s\sqrt{B_{\sigma}/h}}u_{p,\breve\gamma_{h,\sigma}}\big( B_{\sigma}^{1/2}h^{-1/2}t;\xi\big)\Big|^2\psi_{h}^2(s,\sigma) dsdt  \\
&\qquad
+hB_{\sigma}\breve\gamma_{h,\sigma}|u_{p,\breve\gamma_{h,\sigma}}\big( 0;\xi\big)|^2  \Big\}\\
&\qquad+{(1+\norm{k}_{\infty}\delta(h))\Big((1+\varepsilon)Ch^2\delta(h)^{-2}+C \varepsilon^{-1}\delta(h)^4\Big) } \\
&\qquad+     hB_{\sigma}^{1/2}\big|u_{p,\breve\gamma_{h,\sigma}}\big(0;\xi\big)\big|^2 \int_{\R}(\gamma(s)-\gamma_{a}(s))_{+}|\psi_{h}(s;\sigma)|^2ds     \\
&\qquad \leq (1+\varepsilon)hB_{\sigma}\mu_{p}(\breve\gamma_{h,\sigma},\xi)
+{(1+\norm{k}_{\infty}\delta(h))\Big((1+\varepsilon)Ch^2\delta(h)^{-2}+C \varepsilon^{-1}\delta(h)^4\Big) }\\
&\qquad+     hB_{\sigma}^{1/2}|u_{p,\breve\gamma_{h,\sigma}}\big(0;\xi\big)|^2 \int_{\R}(\gamma(s)-\gamma_{a}(s))_{+}|\psi_{h}(s;\sigma)|^2ds \Big\}\,.
\end{aligned}
\end{equation}
Using Lemma~\ref{lem-lim-mu-j} and the fact that $\gamma\in
L^{\infty}$, we infer that the number of indices $p$ appearing in
the support of $M(h,\sigma,\xi,p,K)$ is finite. More precisely,
there exists a constant $p_0\in\mathbb N$ such that, for all
$\sigma>0$, $a\in(0,1)$ and $K>0$, the function $M$ in \eqref{eq:M}
vanishes for all $p> p_0$.

Now, we collect \eqref{Eq:deftr}, \eqref{norm:fj:0},
\eqref{norm-fk-lb} and \eqref{Eqt:Qfj} to obtain
\begin{equation}\label{Eq:tr:alpha=1/2;}
\begin{aligned}
{\rm tr}[(\mathcal{P}_{h,\Omega}^{\alpha,\gamma}-\lambda h)]
&\leq (2\pi)^{-1}h^{-1/2}\iint B_{\sigma}^{1/2}\sum_{p=1}^{\infty} \Big\{\mathcal{Q}_{h,\Omega}^{\alpha,\gamma}(f_{p}(x;h,\sigma,\xi))-\lambda h\norm{f_{p}(x;h,\sigma,\xi)}^2 \Big\}d\xi d\sigma\\
&\leq \Sum{p=1}{p_0}\iint (2\pi)^{-1}B_{\sigma}^{1/2}h^{-1/2}M(h,\sigma,\xi,p,K)\Big\{(1+\varepsilon)hB_{\sigma}\mu_{p}(\breve\gamma_{h,\sigma},\xi)hB_{\sigma}\\
&-\lambda h (1-\norm{k}_{\infty}\delta(h))\big(1-C_{\epsilon,K} e^{-\epsilon(\frac{1}{8} bh^{-1}\delta(h)^{2}-2K^2)/2}(1+a^{-2}\delta(h)+(a^{-2}\delta(h))^2)\big)\\
&\qquad\qquad    +(1+\norm{k}_{\infty}\delta(h))\big((1+\varepsilon)C{h^{2}}{\delta(h)^{-2}}+C\varepsilon^{-1}\delta(h)^{4}\big)\\
&\qquad\qquad    +     hB_{\sigma}^{1/2}|u_{p,\breve\gamma_{h,\sigma}}\big(0;\xi\big)|^2 \int_{\R}(\gamma(s)-\gamma_{a}(s))_{+}|\psi_{h}(s;\sigma)|^2ds  \Big\}{d\sigma d\xi}\,.
\end{aligned}
\end{equation}
We may arrange the terms in \eqref{Eq:tr:alpha=1/2;} to obtain,
\begin{equation}\label{Eq:tr:alpha=1/2}
\begin{aligned}
&{\rm tr}[(\mathcal{P}_{h,\Omega}^{\alpha,\gamma}-\lambda h)]
\quad\leq -\Sum{p=1}{p_0}
\Int{-K}{K}\Int{0}{|\partial\Omega|} (2\pi)^{-1}B_{\sigma}^{1/2}h^{-1/2}\Big(\mu_{p}(\breve\gamma_{h,\sigma},\xi)hB_{\sigma}-\lambda h\Big)_{-}    {d\sigma d\xi}\\
&\qquad + R_1+R_2+R_3\,,
\end{aligned}
\end{equation}
where
\begin{multline}
R_1=
  p_{0}\,h|\partial\Omega|\,2K{\norm{B}}^{3/2}_{L^{\infty}(\partial\Omega)}(2\pi)^{-1}h^{-1/2}\times\\
  \Big(\varepsilon+\norm{k}_{\infty}\delta(h)+ C_{\epsilon,K} e^{-\epsilon(\frac{1}{8} bh^{-1}\delta(h)^{2}-2K^2)/2}(1+a^{-2}\delta(h)+(a^{-2}\delta(h))^2)\\-\norm{k}_{\infty}\delta(h) C_{\epsilon,K} e^{-\epsilon(\frac{1}{8} bh^{-1}\delta(h)^{2}-2K^2)/2}(1+a^{-2}\delta(h)+(a^{-2}\delta(h))^2)\Big)\,,\label{eq:R1}
  \end{multline}
\begin{equation}
R_2=p_{0}\,|\partial\Omega|\,2K{\norm{B}}^{1/2}_{L^{\infty}(\partial\Omega)}(2\pi)^{-1}h^{-1/2}(1+\norm{k}_{\infty}\delta(h))
  \Big((1+\varepsilon){h^{2}}{\delta(h)^{-2}}+C\varepsilon^{-1}\delta(h)^{4} \Big)
    \,,\label{eq:R2}
    \end{equation}
    and
\begin{equation}\label{eq:R3}
R_3=
 (2\pi)^{-1} h^{1/2}{\norm{B}}_{L^{\infty}(\partial\Omega)}\sum_{p=1}^{p_0}
 \int_{-K}^{K}\int_{0}^{|\partial\Omega|}\int_{\R}(\gamma(s)-\gamma_{a}(s))_{+}|\psi_{h}(s;\sigma)|^2|u_{p,\breve\gamma_{h,\sigma}}(0;\xi)|^2 ds d\sigma d\xi\,.
\end{equation}
Choosing $\delta=h^{3/8}$ and $\varepsilon=h^{1/4}$, we see that,
for fixed $a$ and $K$,
\begin{equation}\label{eq:R1+R2}
R_1+R_2=o(h^{1/2})\,,
\end{equation} and
$$|R_3|\leq (2\pi)^{-1}  2K h^{1/2}
{\norm{B}}_{L^{\infty}(\partial\Omega)}\norm{\gamma-\gamma_{a}}_{L^{1}(\partial\Omega)}\sum_{p=1}^{p_0}\sup_{\xi\in[-K,K]}|u_{p,\breve\gamma_{h,\sigma}}(0;\xi)|^2\,.$$
The term $|u_{p,\breve\gamma_{h,\sigma}}(0;\xi)|^2$ is controlled by
the estimate in Lemma~\ref{Lem:|u0|^2}. Taking into account the
condition of the support of $M=M_{1/2}$ in \eqref{eq:M}, we observe
that,
$$|u_{p,\breve\gamma_{h,\sigma}}(0;\xi)|^2\leq C\Big(\mu_p(\breve\gamma_{h,\sigma};\xi)+\breve\gamma_{h,\sigma}^2 +1 \Big)\leq C(2+\breve\gamma_{h,\sigma}^2).$$
It follows from the definition of  $\breve\gamma_{h,\sigma}$ in
\eqref{def:gamma-tilde} that, when $h$ and $\sigma$ vary and $K$,
$a$ and $p$ remain fixed,
$$\sup_{\xi\in[-K,K]}|u_{p,\breve\gamma_{h,\sigma}}(0;\xi)|^2\leq C\left(1+(a^{-2}\delta(h))^2\right)\,,$$
thereby giving us that,
\begin{equation}\label{eq:R3'}
|R_3|\leq
2CKh^{1/2}\left(1+(a^{-2}\delta(h))^2\right)\|\gamma-\gamma_a\|_{L^1(\partial\Omega)}\,,
\end{equation}
as long as $a$ and $K$ remain fixed.

Now, we insert \eqref{eq:R1+R2} and \eqref{eq:R3'} into
\eqref{Eq:tr:alpha=1/2}. Thanks to Lemma~\ref{Pro-dm} below, we may
apply the variational principle in Lemma~\ref{lem-VP-3}. That way,
we infer from \eqref{Eq:tr:alpha=1/2},
\begin{equation}\label{Eq:tr:alpha=1/2}
\begin{aligned}
 &-E(\lambda;h,\gamma,1/2)
\leq  (1+\norm{k}_{\infty}\delta(h))^{-1} {\rm tr}[(\mathcal{P}_{h,\Omega}^{\alpha,\gamma}-\lambda h)]
 \\
 &\qquad\leq - (1+\norm{k}_{\infty}\delta(h))^{-1}\Sum{p=1}{p_0}\Int{-K}{K}\Int{0}{|\partial\Omega|} (2\pi)^{-1}B_{\sigma}^{3/2}h^{1/2}\Big(\mu_{p}(\breve\gamma_{h,\sigma},\xi)-\frac{\lambda}{B_{\sigma}} \Big)_{-}    {d\sigma d\xi}\\
&\qquad\qquad    +o(h^{1/2})+2CKh^{1/2}\left(1+(a^{-2}\delta(h))^2\right)
    \norm{\gamma-\gamma_{a}}_{L^{1}(\partial\Omega)}\,.
\end{aligned}
\end{equation}
 Since ${\breve\gamma}_{h,\sigma}\rightarrow B_{\sigma}^{-1/2}
{\gamma_a(\sigma)}$ as $h\rightarrow 0$, and
${\breve\gamma}_{h,\sigma}$ remains bounded for a fixed $a$, then it
results from Lemma~\ref{lim-int} and dominated convergence (as
$h\to0_+$),
\begin{equation}
\sum_{p=1}^{p_0}\int_{\R}\Big(\mu_{p}(\breve\gamma_{h,{\sigma}}\,,\,\xi)-\frac{\lambda}{B_{\sigma}}\Big)_{-}d\xi\to
\sum_{p=1}^{p_0}\int_{\R}\Big(\mu_{p}\big(B_{\sigma}^{-1/2}{\gamma_a(\sigma)},\xi\big)
-\frac{\lambda}{B_{\sigma}}\Big)_{-}d\xi.
\end{equation}
Since the  function $\gamma_a$ is smooth and bounded (for every
fixed $a$), then by dominated convergence,
$$\int_{0}^{|\partial\Omega|}
\sum_{p=1}^{p_0}\int_{\R}\Big(\mu_{p}(\breve\gamma_{h,{\sigma}}\,,\,\xi)-\frac{\lambda}{B_{\sigma}}\Big)_{-}d\xi\,d\sigma
\to \int_{0}^{|\partial\Omega|}
\sum_{j=1}^{p_0}\int_{\R}\Big(\mu_{p}\big(B_{\sigma}^{-1/2}{\gamma_a(\sigma)}
,\xi\big) -\frac{\lambda}{B_{\sigma}}\Big)_{-}d\xi d\sigma.$$ Taking
$\limsup_{h\to 0}$ on both sides in \eqref{Eq:tr:alpha=1/2}, it
follows that,
\begin{multline*}
    \limsup_{h\to 0}\Big(-h^{-1/2}E(\lambda;h,\gamma,1/2)\Big)\\
    \leq -\Sum{p=1}{p_0}\Int{-K}{K}\int_{\partial\Omega} (2\pi)^{-1}B(x)^{3/2}\Big(\mu_{p}(B(x)^{-1/2}{ \gamma_{a}(x)},\xi)-\dfrac{\lambda}{B(x)} \Big)_{-}    { d\xi ds(x)} \,.
\end{multline*}
Now, we take the successive limits, $\limsup_{a\rightarrow 0_{+}}$
and $\lim_{K\rightarrow\infty}$ to obtain,
\begin{multline}\label{Eq:ub}
    \limsup_{h\to 0}\Big(-h^{-1/2}E(\lambda;h,\gamma,1/2)\Big) \\
    \leq -\lim_{K\rightarrow\infty}\Bigg\{\liminf_{a\rightarrow 0_{+}}\frac{1}{2\pi}\Sum{p=1}{p_0}\Int{-K}{K}\int_{\partial\Omega} B(x)^{3/2}\Big(\mu_{p}(B(x)^{-1/2}{ \gamma_{a}(x)},\xi)-\frac{\lambda}{B(x)}\Big)_{-}    {d\xi ds(x)} \Bigg\}\,.
\end{multline}
Since $\gamma\in L^{\infty}(\Omega)$, then $\|\gamma_{a}\|_{\infty}\leq \|\gamma\|_{\infty}$ and by dominated convergence, the right-hand side in \eqref{Eq:ub} is
\[
-\dfrac{1}{2\pi}\Sum{p=1}{p_0}\Int{-\infty}{\infty}\int_{\partial\Omega}B(x)^{3/2}\left(\mu_{p}(B(x)^{-1/2}{ \gamma(x)},\xi)-\frac{\lambda}{B(x)} \right)_{-}    { d\xi ds(x)}\,.
\]
This finishes the proof of the upper bound in Theorem~\ref{thm:KN}.

It remains to verify that  the density matrix $\Gamma$ satisfies the
necessary properties to apply the variational principle in
Lemma~\ref{lem-VP-3}. That is contained in

\begin{lem}\label{Pro-dm}
There exists a constant $C>0$ such that,
\begin{equation}\label{DM-cond}
\forall~f\in L^2(\Omega)\,,\quad
0\leq \langle \Gamma f, f\rangle_{L^{2}(\Omega)} \leq (1+\norm{k}_{\infty}\delta(h))\norm{f}^{2}_{L^{2}(\Omega)},
\end{equation}
where $\Gamma$ is as in \eqref{eq:Gamma}.
\end{lem}
\begin{proof}
Let $f\in L^{2}(\Omega)$. Due to the support of $\Gamma$ (in particular $\zeta_{1,h}$), we may suppose that ${\rm supp}~{f}\subset\{ x\in\overline{\Omega}~:~{\rm dist}(x,\partial\Omega)\leq \delta(h)\}.$
We compute,
\begin{multline}\label{eq1-g-gamma-g}
    \langle f, \Gamma f\rangle_{L^{2}(\Omega)}\\
    =({2\pi})^{-1}h^{-1/2}\Sum{p=1}{\infty}\iint M(h,\sigma,\xi,p,K)B_{\sigma}^{1/2}\left|\iint \overline{\widetilde f(s,t)}{\widetilde f}_{p}((s,t); h,\sigma,\xi)(1-tk(s)) ds dt    \right|^{2}{d\sigma
    d\xi}\,.
\end{multline}
We estimate from above by replacing $\iint M\times|\cdot|^{2}$ by $\iint1\times|\cdot|^{2}$ in the above expression.
Defining
\[
G(s,t)=\overline{\widetilde f(s,t)}\psi_{h}(s;\sigma)\zeta_{1,h}(t)e^{-i\phi_{\sigma}/h}(1-tk(s)).
\]
Using Cauchy-Schwarz inequality and the fact that
$u_{p,\breve\gamma_{h,\sigma}}(\cdot,\xi)$ in the case $\alpha=1/2$
(or $u_{p,0}(\cdot,\xi)$ in the case $\alpha>1/2$) is an orthonormal
basis of $L^{2}(\R_{+})$ for all $\xi$, we get,
\begin{equation*}
\sum_{p=1}^{\infty}\left|\iint \overline{\widetilde f(s,t)}\widetilde f_{p}(s,t;h,\sigma,\xi)(1-tk(s))dsdt\, \right|^{2}\leq 2\pi\Int{t>0}{}\left|(\mathcal{F}_{s\rightarrow\xi}G)({B_{\sigma}^{1/2}h^{-1/2}}\xi)\right|^{2}dt\,.
\end{equation*}
Here, $\mathcal{F}_{s\rightarrow\xi}$ denotes the Fourier transform with respect to the variable $s$.

Integrating with respect to $\xi$ and using Plancherel identity, we
find that,
\begin{align*}
&2\pi\Int{\R}{}\Int{t\in \R_{+}}{}\left|(\mathcal{F}_{s\rightarrow\xi}G)(B_{\sigma}^{1/2}h^{-1/2}\xi)\right|^{2}dtd\xi=
 2\pi h^{1/2}B_{\sigma}^{-1/2}\Int{\R^{2}_{+}}{}|G(s,t)|^{2}dsdt.
\end{align*}
Consequently,
\begin{equation}\label{eq2-g-gamma-g}
    \langle f, \Gamma f\rangle_{L^2(\Omega)}\leq \Int{0}{|\partial\Omega|}\int_{\R^{2}_{+}}|\widetilde f(s,t)|^{2}(1-tk(s))^{2}\psi_{h}^{2}(s;\sigma)\zeta_{1,h}^{2}(t)dsdtd\sigma.
\end{equation}
We do the $\sigma$-integration first. The normalization of $\psi_{h}$ implies that the result is
\begin{align*}
    \langle f, \Gamma f\rangle_{L^{2}(\Omega)}&\leq \Int{\R^{2}_{+}}{}|\widetilde f(s,t)|^{2}(1-tk(s))^{2}\zeta_{1,h}^{2}(t)dsdt\\
    &\leq (1+\delta(h)\norm{k}_{\infty})\Int{\Omega}{}|f(x)|^{2}dx.
\end{align*}
This finishes the proof of \eqref{DM-cond}.
\end{proof}
\section{Proof of Corollary~\ref{cor:KN}}\label{Sec:7}
We will prove the second assertion in \eqref{SA}. The first
assertion in \eqref{FA} can be proven similarly. Define
\begin{equation}\label{Def:fp}
f_{p}(x,\lambda):=\int_{\R}B(x)^{3/2}\Big(\mu_{p}\left(B(x)^{-1/2}\gamma(x),
\xi\right)-\frac{\lambda}{B(x)}\Big)_{-}d\xi\,.
\end{equation}
We start by computing the left- and right- derivatives of the
function $ \lambda\to f_{p}(x,\lambda)$.  We thus find
\begin{equation}\label{right-der}
\dfrac{\partial f_{p}}{\partial\lambda_{+}}(x,\lambda)=\int_{
\{\xi\in\R~:~B(x)\mu_p\left(B(x)^{-1/2}\gamma(x),\xi\right)\leq\lambda\}}
B(x)^{1/2}d\xi\,,
\end{equation}
and
\begin{equation}
\dfrac{\partial f_{p}}{\partial\lambda_{-}}(x,\lambda)=\int_{
\{\xi\in\R~:~B(x)\mu_p\left(B(x)^{-1/2}\gamma(x),\xi\right)<\lambda\}}
B(x)^{1/2}d\xi.
\end{equation}

In view of Lemma~\ref{lem:Ka}, the equation
\[
\mu_{p}\big(B(x)^{-1/2}\gamma(x),\xi\big)={\lambda}B(x)^{-1},
\]
has exactly two solutions
\[
\xi_{p,\pm}:=\xi_{p,\pm}(\gamma^{\prime}_{x},\lambda^{\prime}_{x});\quad \gamma^{\prime}_{x}:=B(x)^{-1/2}\gamma(x)\quad{\rm and}\quad \lambda^{\prime}_{x}:=\lambda B(x)^{-1}.
\]
Since the set $\{\xi: \xi=\xi_{p,\pm}\}$ has measure zero with respect to the $\xi$ integration, it follows that the left- and right- derivatives coincide and we can write
\begin{equation}\label{Def:derfp}
\dfrac{\partial f_{p}}{\partial\lambda_{\pm}}(x,\lambda)=\int_{
\{\xi\in\R~:~B(x)\mu_p\left(B(x)^{-1/2}\gamma(x),\xi\right)\leq\lambda\}}
B(x)^{1/2}d\xi.
\end{equation}
Let $\varepsilon>0$. By the variational principle in
Lemma~\ref{lem-VP-3}, we have
\begin{equation}{\label{First-eq}}
E(\lambda+\varepsilon;h,\gamma,\alpha) -E(\lambda;h,\gamma,\alpha) \geq \varepsilon h {N}(\lambda;h,\gamma,\alpha).
\end{equation}
On the other hand, it follows from Theorem~\ref{thm:KN} that
\begin{equation}\label{F:thKN}
E(\lambda;h,\gamma,\alpha)
=\frac{h^{1/2}}{2\pi}\sum_{p=1}^{\infty}\int_{\partial\Omega}f_{p}(x,\lambda)ds(x)+h^{1/2}o(1).
\end{equation}
In light of Lemma~\ref{lem-lim-mu-j}, the sum on the right hand side of \eqref{F:thKN} is actually a sum of a finite number of terms. Thus $\sum_{p=1}^{\infty}$ can be replaced by $\sum_{p=1}^{p_0}$.
Implementing \eqref{F:thKN} into \eqref{First-eq}, then taking $\limsup_{h\rightarrow 0_{+}}$, we get
\begin{equation}
 \limsup_{h\rightarrow 0_{+}}h ^{1/2} {N}(\lambda;h,\gamma,\alpha)\leq\dfrac{1}{2\pi} \sum_{p=1}^{p_0}\int_{\partial\Omega}\dfrac{f_{p}(x,\lambda+\varepsilon)-f_{p}(x,\lambda)}{\varepsilon} d\sigma(x).
\end{equation}
Taking the limit $\varepsilon\rightarrow 0_{+}$ and using dominated
convergence, we deduce that
\begin{equation}\label{u-b-n}
\limsup_{h\rightarrow 0_{+}}  h^{1/2}{N}(\lambda;h,\gamma,\alpha)\leq \dfrac{1}{2\pi} \sum_{p=1}^{p_0}\int_{\partial\Omega}\dfrac{\partial f_p}{\partial\lambda_+}(x,\lambda)d\sigma(x).
\end{equation}
Replacing $\varepsilon$ by $-\varepsilon$ in \eqref{First-eq} and following the same arguments that led to \eqref{u-b-n}, we find
 \begin{equation}\label{l-b-n}
\liminf_{h\rightarrow 0_{+}}  h^{1/2}{N}(\lambda;h,\gamma,\alpha)\geq \dfrac{1}{2\pi} \sum_{p=1}^{p_0}\int_{\partial\Omega}\dfrac{\partial f_p}{\partial\lambda_-}(x,\lambda)d\sigma(x).
\end{equation}
By combining \eqref{u-b-n} and \eqref{l-b-n}, we obtain
\begin{equation}\label{F-eq-nb}
\lim_{h\rightarrow 0_{+}}  h^{1/2}{N}(\lambda;h,\gamma,\alpha)= \dfrac{1}{2\pi} \sum_{p=1}^{p_0}\int_{\partial\Omega}\dfrac{\partial f_p}{\partial\lambda_\pm}(x,\lambda)d\sigma(x)\,.
\end{equation}
Now, in light of \eqref{Def:derfp}, we finally get that
\begin{multline}\label{F-eq-nb}
\lim_{h\rightarrow 0_{+}}  h^{1/2}{N}(\lambda;h,\gamma,\alpha)
=\frac{1}{2\pi}\sum_{p=1}^\infty
\iint_{
\{(x,\xi)\in\partial\Omega\times\R~:~B(x)\mu_p\left(B(x)^{-1/2}\gamma(x),\xi\right)<\lambda\}}
B(x)^{1/2}d\xi ds(x)\,.
\end{multline}
This finishes the proof of \eqref{SA}.
\section{Proof of Theorem~\ref{thm:SQ}}

We will apply a simple scaling argument to pass from the
semi-classical to the large area limit. Let $T$ be a positive number
and $\Omega_T=(0,T)\,\times\,(0,T)$. Define the operator,
$$P_{\Omega_T}=-(\nabla-i\Ab_0)^2\quad {\rm in}~L^2(\Omega_T)\,.$$
Functions  in the domain of $P_{\Omega_T}$ satisfy Neumann condition
$\nu\cdot(\nabla-i\Ab_0)u=0$ on the smooth parts of the boundary of
$\Omega_T$. We assume that the vector field $\Ab_0$ is given by
\begin{equation}\label{eq:vf}
\Ab_0(x_1,x_2)=(-x_2,0)\,,\quad\Big((x_1,x_2)\in\mathbb R^2\Big)\,.
\end{equation}
The operator $P_{\Omega_T}$ has compact resolvent and its spectrum
consists of an increasing sequence of eigenvalues $(e_j)_{j\geq1}$
converging to $\infty$. Note that the terms of the sequence $(e_j)$
are listed counting   multiplicities. Given $\lambda\geq 0$, the
number of eigenvalues below $1+\lambda$ is finite. Denote by
\begin{equation}\label{nb-torus}
\mathcal N(\lambda,T)={\rm Card}\,\{j~:~e_j\leq 1+\lambda\}\,.
\end{equation}

By a scaling argument, Theorem~\ref{thm:SQ} follows from:

\begin{thm}\label{thm:TDL}
There exists a positive number $\delta$ such that,
$$\limsup_{T\to\infty}\frac{\mathcal
N(\lambda,T)}{T^2}=\frac1{2\pi}\,,\quad(\lambda\in[0,\delta])\,.$$
\end{thm}

\subsection{Preliminaries}\label{sec:p}

\subsubsection{Variational min-max principle}\label{sec:vp}

We shall need the following version of the variational min-max
principle.

\begin{thm}\label{thm:vp}
Let $A$ be a self-adjoint operator in a Hilbert space $H$. Suppose
that $A$ is semi-bounded (i.e. bounded from below) and has compact
resolvent. The terms of the sequence of eigenvalues of $A$ counting
multiplicities are given by,
$$\mu_n=\inf\Big\{\max_{\substack{\phi\in M\\\|\phi\|_H=1}}\langle
A\phi,\phi\rangle_H~:~M\subset D(A)\,,~{\rm dim}\,M=n\Big\}\,.$$
\end{thm}

\subsubsection{Rough  bound for the operator
$P_{\Omega_T}$}\label{sec:torus}

Let $S$ and $T$ be positive numbers, and
$\Omega_{S,T}=(0,S)\times(0,T)$. Consider the operator,
$$P_{\Omega_{S,T}}=-(\nabla-i\Ab_0)^2\quad{\rm
in}~L^2(\Omega_{S,T})\,.$$ A function $u(x_1,x_2)$ in the domain of
$P_{\Omega_{S,T}}$ satisfies Neumann condition at $x_2=0$, Dirichlet
condition at $x_2=T$, and periodic conditions at $x_1\in\{0,S\}$.
Define,
$$\mathcal N(\lambda;S,T)={\rm tr}\Big(\mathbf
1_{(-\infty,(1+\lambda)bh]}(P_{h,b,\Omega_{S,T}})\Big)\,.$$

Along the proof of Lemma~3.1 in \cite{Fo-Ka}, a useful rough bound on
$\mathcal N(\lambda;S,T)$ is given. We recall this bound below.

\begin{lem}\label{roughestimate'}
There exist positive constants $C$, $T_0$ and $\lambda_0$ such that,
for all $T\geq T_0$, $\lambda\in[0,\lambda_0]$ and $S>0$, we have,
\begin{equation}\label{eq-cyl:nb}
\mathcal N(\lambda;S,T)\leq  CST \,.
\end{equation}
\end{lem}

\subsubsection{The Dirichlet operator in a
square}\label{sec:Dirchlet}

Recall the magnetic potential $\Ab_0$ in \eqref{eq:vf}. Consider a
positive real number $R$ and the operator
$P^D_{\Omega_R}=-(\nabla-i\Ab_0)^2$ in the square
$\Omega_R=(0,R)\times(0,R)$ and with Dirichlet boundary conditions.
If $\Lambda\in\R$, we define the functions,
\begin{equation}\label{eq-nub}
\nu_{b}(\Lambda)=\frac{1}{2\pi}{\rm Card}\,\left\{n\in\N~:~2n-1\leq\Lambda\right\}\,.
\end{equation}
and
\begin{equation}\label{eq-nub}
N\left(\Lambda ,P^D_{\Omega_R}\right)={\rm tr}\Big(\mathbf 1_{(-\infty,\Lambda ]}(P^D_{\Omega_R})\Big)\,.
\end{equation}
 The next two-sided estimate on the eigenvalue counting function
of the operator $P^D_{\Omega_R}$  is proved in \cite[Thm.~3.1]{CdV}.

\begin{lemma}\label{lem-CdV}
There exists a constant $C>0$ such that, for all $\Lambda\in\R$,
$R>0$  and $A\in(0,R/2)$, the following two-sided estimate holds
true,
$$
(R-A)^2\nu_{b}\left(\Lambda-\frac{C}{A^2}\right)\leq
N\left(\Lambda ,P^D_{\Omega_R}\right)\leq R^2\nu_{b}(\Lambda)\,.
$$
In particular, if $\Lambda<3$, then
$$N\left(\Lambda ,P^D_{\Omega_R}\right)\leq
\frac{R^2}{2\pi}\,.$$
\end{lemma}

\subsubsection{The periodic operator}\label{sec:per}

Consider a positive number $R$, the square
$\Omega_R=(0,R)\times(0,R)$ and the function space,
\begin{equation}\label{eq:ER}
E_R=\{u\in H^1_{\rm loc}(\R^2)~:~u(x_1+R,x_2)=u(x_1,x_2)~\&~u(x_1,x_2+R)=e^{-iRx_1}u(x_1,x_2)\,\}\,.
\end{equation}
Recall the magnetic potential $\Ab_0$ in \eqref{eq:vf}. If $u\in
E_R$, then $|u|$ and $|(\nabla-i\Ab_0)u|$ are periodic with respect
to the lattice generated by $\Omega_R$. Consider the self-adjoint
operator
$$P^{\rm per}_{\Omega_R}=-(\nabla-i\Ab_0)^2\quad{\rm in}\quad
L^2(\Omega_R)\,,$$ whose domain is that defined by the Friedrichs'
extension associated with the quadratic form,
$$E_R\ni f\mapsto
\int_{\Omega_R}|(\nabla-i\Ab_0)u|^2\,dx\,.$$ Denote by $(\mu_j)$ the
sequence of distinct eigenvalues of the operator $P^{\rm
per}_{\Omega_R}$. Let us recall the following classical results (see
\cite[Proposition~2.9]{FK3D}). These results are valid under the
assumption that $R^2/(2\pi)$ is a positive integer.
\begin{itemize}
\item The first eigenvalue of $P^{\rm per}_{\Omega_R}$ is $\mu_1(P^{\rm per}_{\Omega_R})=1$ and the second eigenvalue $\mu_2(P^{\rm per}_{\Omega_R})\geq3$.
\item The dimension of the eigenspace ${\rm Ker}(P^{\rm
per}_{\Omega_R}-{\rm Id})$ is $R^2/(2\pi)$\,.
\end{itemize}

As a consequence, we may state the following lemma.

\begin{lem}\label{lem:per}
Suppose that $R^2\in 2\pi\mathbb N$. If $0\leq\lambda<2$ and $ N_{\rm
per}(\lambda,R)={\rm tr}\big(\mathbf
1_{(-\infty,1+\lambda]}(P_{\Omega_R})\big)$, then,
$$ N_{\rm
per}(\lambda,R)=\frac{R^2}{2\pi}\,.$$
\end{lem}

\subsubsection{The operator in a sector}\label{sec:sec}

Recall the magnetic potential $\Ab_0$ in \eqref{eq:vf}. Consider
the operator,
\begin{equation}\label{eq:sec}
P_{\Omega_{R,\pi/2}}=-(\nabla-i\Ab_0)^2\quad {\rm in}~L^2(\Omega_{R,\pi/2})\,,
\end{equation}
where $\Omega_{R,\pi/2}=\{(r\cos\theta,r\sin\theta)~:~0\leq
r<R\,,~0<\theta<\pi/2\,\}$. Functions in the domain of
$P_{\Omega_{R,\pi/2}}$ satisfy Neumann condition on $\theta=0$ and
$\theta=\pi/2$, and Dirichlet condition on $r=R$.

The operator $P_{\Omega_{R,\pi/2}}$ has compact resolvent and its
spectrum consists of an increasing sequence of eigenvalues $(\zeta_j)$
counting multiplicities. We introduce,
\begin{equation}\label{eq:nb-sec}
\mathcal N_{\rm sec}(\lambda,R)={\rm Card}\{j~:~\zeta_j\leq 1+\lambda\}\,.
\end{equation}

A useful rough bound on $\mathcal N_{\rm sec}(\lambda,R)$ is proved
in \cite{KK}. We recall this bound in the next lemma.

\begin{lem}\label{lem:sec}
There exist positive constants $C$, $R_0$ and $\lambda_1$ such that,
for all $R\geq R_0$ and $\lambda\in[0,\lambda_1]$, we have,
$$\mathcal N_{\rm sec}(\lambda,R)\leq C(R^2+1)\,.$$
\end{lem}

\subsection{Proof of Theorem~\ref{thm:TDL}}\label{sec:TDL} Through
this section, the following convention will be used. If $P$ is a
self-adjoint operator and $\Lambda<\inf\sigma_{\rm ess}(P)$, denote
by
$$ N(\Lambda, P)={\rm tr}\Big(\mathbf
1_{(-\infty,\Lambda]}(P)\Big)\,.$$

Recall the operator $P_{\Omega_T}$ and the number $\mathcal
N(\lambda, T)=N\big(1+\lambda,P_{\Omega_T}\big)$ introduced in
\eqref{nb-torus}.

We start by the observation:

\begin{lem}\label{lem:TDL-per}
Let $T_n=\sqrt{2\pi\, n}$, $n\in\mathbb N$. For all
$\lambda\in[0,2)$, there holds,
$$\frac{\mathcal N(1+\lambda,T_n)}{T_n^2}\geq \frac1{2\pi}\,.$$
\end{lem}
\begin{proof}
Recall the operator $P^{\rm per}_{\Omega_R}$ introduced in
Sec.~\ref{sec:per} together with the number $\mathcal N_{\rm per}(\lambda,R)$
in Lemma~\ref{lem:per}. Notice that  functions in the form domain of
$P^{\rm per}_{\Omega_R}$ are in $H^1(\Omega_R)$ and consequently in
the form domain of $P_{\Omega_R}$. The variational min-max principle
(Theorem~\ref{thm:vp}) then tells us that the eigenvalues of $P^{\rm
per}_{\Omega_R}$ are larger than the corresponding ones of
$P_{\Omega_R}$. Consequently (we use $R=T_n$),
$$\mathcal N(\lambda,T_n)\geq N_{\rm
per}(\lambda,T_n)\,.$$ Notice that $T_n^2\in2\pi\mathbb N$. Consequently, when $\lambda\in[0,2)$, it results from Lemma~\ref{lem:per} that
$$N_{\rm
per}(\lambda,T_n)=\frac{T_n^2}{2\pi}\,.$$ This proves Lemma~\ref{lem:TDL-per}.
\end{proof}

\begin{lem}\label{lem:subseq}
There exist positive constants $\delta\in(0,1)$, $T_1$ and $C$ such
that, for all $\lambda\in[0,\delta]$ and $T\geq T_1$, there holds,
$$\mathcal N(\lambda,T)\leq \frac{T^2}{2\pi}+CT\,.$$
\end{lem}
\begin{proof}
Consider a number $L\in(0,T)$. We cover the square
$\Omega_T=(0,T)\times(0,T)$ by sets $U$, $V_j$ and $U_j$,
$j\in\{1,2,3,4\}$ defined as follows:
\begin{align*}
&U=\left(\frac{L}2,T-\frac{L}2\right)\times \left(\frac{L}2,T-\frac{L}2\right)\,,\\
&U_1=\left(\frac{L}2,T-\frac{L}2\right)\times[0,L)\,,& U_2&=[0,L)\times \left(\frac{L}2,T-\frac{L}2\right)\,,\\
&U_3=\left(\frac{L}2,T-\frac{L}2\right)\times(T-L,T]\,,& U_4&=(T-L,T]\times\left(\frac{L}2,T-\frac{L}2\right)\,,\\
&V_1=[0,L)\times[0,L)\,,& V_2&=(T-L,T]\times [0,L)\,,\\
&V_3=[0,L)\times (T-L,T]\,,& V_4&=(T-L,T]\times(T-L,T]\,.
\end{align*}
Let $P_{V_j}$ and $P_{U_j}$ be   self-adjoint realizations  of the
operator $-(\nabla-i\Ab_0)^2$ in $L^2(V_j)$ and $L^2(U_j)$
respectively and defined as follows.  For every $j$ and
$\Omega\in\{V_j,U_j\}$, functions in the domain of $P_{\Omega}$
 satisfy Neumann condition on the common smooth
boundary of $\Omega$ and $\Omega_T$ and Dircihlet condition
elsewhere.

Notice that the operators $P_{V_j}$, $j\in\{1,2,3,4\}$, are unitary
equivalent and have the same spectra. Also, it results from the
variational min-max principle that the spectrum of $P_{V_j}$ is
below that of the operator $P_{\Omega_{2L,\pi/2}}$ introduced in
Sec.~\ref{sec:sec} thereby obtaining,
$$\mathcal N(\lambda,P_{V_j})\leq \mathcal N_{\rm sec}(\lambda,2R)\,.$$
The operators $P_{U_j}$, $j\in\{1,2,3,4\}$,  are unitary equivalent
also and  (recall the operator $P_{\Omega_{S,T}}$ introduced in
Sec.~\ref{sec:torus}),
$$\sigma(P_{U_j})=\sigma(P_{\Omega_{T-L,L}})\,,\quad(j\in\{1,2,3,4\})\,.$$
Consider a partition of unity
$$\sum_{j=1}^4\chi_j^2+\sum_{j=1}^4\varphi_j^2+f^2=1\quad{\rm
in~}\overline{\Omega_T}\,,$$ such that
$$\sum_{j=1}^4\left(|\nabla\chi|^2+|\nabla\varphi_j|^2\right)+|\nabla
f|^2\leq \frac{C}{L^2}\,,$$
$$
{\rm supp}\chi_j\subset V_j\,,\quad {\rm supp}f_j\subset
U_j\,,\quad{\rm supp}\,f\subset U\,,$$
and $C$ is a universal constant.

Using the IMS decomposition formula, we may write for any function
$u$ in the form domain of $P_{\Omega_T}$,
\begin{align*}
q(u)&=q(fu)
+\sum_{j=1}^4q(\chi_ju)+
\sum_{j=1}^4q(\varphi_ju)
-\int_\Omega\left(|\nabla f|^2+\sum_{j=1}^4(|\nabla \chi_j|^2+|\nabla \varphi_j|^2)\right)|u|^2\,dx\\
&\geq q(fu)
+\sum_{j=1}^4q(\chi_ju)+
\sum_{j=1}^4q(\varphi_ju)-\frac{C}{L^2}\int_\Omega|u|^2\,dx\,,
\end{align*}
where the quadratic form $q$ is defined by,
$$q(v)=\displaystyle\int_\Omega|(\nabla-i\Ab_0)v|^2\,dx\,.$$
As has been proven in \cite{CdV}, it results from the variational
min-max principle (Theorem~\ref{thm:vp}):
\begin{equation}\label{eq:TDL-p}
\mathcal{N}(\lambda,P_{\Omega_T})\leq N(1+\lambda+\frac{C}{L^2},
P^D_{\Omega_{T-L}})+\sum_{j=1}^4 N(1+\lambda+\frac{C}{L^2},
P_{V_j})+\sum_{j=1}N(1+\lambda+\frac{C}{L^2},
P_{U_j})\,.\end{equation} Recall that the operator
$P^D_{\Omega_{R}}$ (with $R=T-L$) has been introduced in
Sec.~\ref{sec:Dirchlet}. Let
$\lambda_2=\frac12\min(\lambda_0,\lambda_1,1)$ where $\lambda_0$ and
$\lambda_1$ are as introduced in Lemmas~\ref{roughestimate'} and
\ref{lem:sec}. Select $L$ such that,
$$L\geq T_0\quad{\rm and}~ \lambda+\frac{C}{L^2}<\min(\lambda_0,\lambda_1,1)\,,\quad
(\lambda\in[0,\lambda_2])\,,$$ where $T_0$ is as in Lemma~\ref{roughestimate'}.

 Consequently, it follows from
Lemma~\ref{lem-CdV} that,
$$N(1+\lambda+\frac{C}{L^2},
P^D_{\Omega_{T-L}})\leq \frac{(T-L)^2}{2\pi}\,.$$ Also, as pointed
earlier and using Lemmas~\ref{roughestimate'} and \ref{lem:sec}, we
get for $\lambda\in[0,\lambda_2]$ and sufficiently large $T$,
\begin{align*}
&N(1+\lambda+\frac{C}{L^2},
P_{V_j})\leq\mathcal N_{\rm sec}(\lambda+\frac{C}{L^2},2L)\leq C(4L^2+1)\,,\\
&N(1+\lambda+\frac{C}{L^2},
P_{U_j})\leq \mathcal N(\lambda;T-L,L)\leq C(T-L)L\,.
\end{align*}
By substituting the above upper bounds into \eqref{eq:TDL-p}, we get
the upper bound in Lemma~\ref{lem:subseq}.
\end{proof}

\begin{proof}[Proof of Theorem~\ref{thm:TDL}]
Let $\lambda\in[0,\delta]$ with $\delta$ as in
Lemma~\ref{lem:subseq}. In light of Lemma~\ref{lem:TDL-per}, we get
$$\limsup_{T\to\infty}\frac{\mathcal N(\lambda,T)}{T^2}\geq
\frac1{2\pi}\,.$$ On the other hand, Lemma~\ref{lem:subseq} tells us that,
$$\limsup_{T\to\infty}\frac{\mathcal N(\lambda,T)}{T^2}\leq
\frac1{2\pi}\,.$$
\end{proof}

\appendix
\section{The quadratic form in \eqref{QF-Gen} is semi-bounded}

By density of smooth functions in $H^1$ and compactness of the
boundary $\partial\Omega$, the semi-boundedness of \eqref{QF-Gen}
follows from:
\begin{lem}\label{lem:app}
Let $\gamma\in L^3(\partial\Omega)$ and $x_0\in\partial\Omega$.
There exist constants $C_0>0$ and $r_0>0$ such that, for all $u\in
C^\infty_0(\overline{B(x_0,r_0)\cap\Omega})$,
$$\|\nabla u\|^2_{L^2(B(x_0,r_0))}+\int_{\overline{B(x_0,r_0)}\cap\partial\Omega}\gamma(x)
|u(x)|^2\,ds(x)\geq -C_0\|u\|^2_{L^2(B(x_0,r_0))}\,.$$
\end{lem}
\begin{proof}
Select $r_0$ sufficiently small such that the coordinate
transformation in \eqref{BC} is defined. Using these coordinates, we
may view the function $u$ as a  function in
$C_0^\infty(\overline{\R^2_+})$. In the same way, we may view the
function $\gamma$ in $L^3(\R)$. We have,
\begin{multline}\label{eq:app}
\|\nabla
u\|^2_{L^2(B(x_0,r_0))}+\int_{\overline{B(x_0,r_0)}\cap\partial\Omega}\gamma(x)
|u(x)|^2\,ds(x)\\
\geq C\int_{\R_+}\int_{\R}\left(|\partial_t u|^2+|\partial _s
u|^2\right)\,ds\,dt+\int_\R\gamma(s)|u(s,0)|^2\,ds\,,\end{multline}
where $C>0$ is a constant.

We have the simple identity,
$$\gamma(s)|u(s,0)|^2=-2\int_0^\infty\gamma(s)\,u(s,t)\,\partial_tu(s,t)\,dt\,.$$
By the Cauchy-Schwarz inequality, we get for all $\epsilon>0$,
$$
\gamma(s)|u(s,0)|^2\geq -\epsilon\int_0^\infty|\partial_t
u(s,t)|^2\,dt-4\epsilon^{-1}\int_0^\infty|\gamma(s)\,u(s,t)|^2\,dt\,.$$ We integrate both sides with respect to $s$ to obtain,
\begin{equation}\label{eq:app1}
\int_\R\gamma(s)|u(s,0)|^2ds\geq -\epsilon\int_{\R_+}\int_\R|\partial_t
u(s,t)|^2\,ds\,dt-4\epsilon^{-1}\int_{\R_+}\int_\R|\gamma(s)\,u(s,t)|^2\,ds\,dt\,.\end{equation}
Since $\gamma\in L^3(\R)$, then the operator
$$v\mapsto\gamma v$$
is $i\frac{d}{ds}$-compact. Thus, for all $a\in(0,1)$, there exists
a constant $b>0$ such that,
$$\int_\R|\gamma(s)\,u(s,t)|^2\,ds\leq
a\int_\R|\partial_su(s,t)|^2\,ds+b\int_\R|u(s,t)|^2\,ds\,.$$
Inserting this into \eqref{eq:app1} and then inserting the resulting
inequality into \eqref{eq:app}, we get,
\begin{multline*}
\|\nabla
u\|^2_{L^2(B(x_0,r_0)}+\int_{\overline{B(x_0,r_0)}\cap\partial\Omega}\gamma(x)
|u(x)|^2\,ds(x)\\
\geq \int_{\R_+}\int_\R\Big((C-\epsilon)|\partial_t
u|^2+(C-4\epsilon^{-1}a)|\partial_su|^2-4\epsilon^{-1}b|u|^2\Big)\,ds\,dt\,.\end{multline*}
We select $\epsilon=C$ and $a$ sufficiently small such that
$C-4\epsilon^{-1}a>0$. Returning to cartesian coordinates, we get
the estimate in Lemma~\ref{lem:app}.
\end{proof}

\end{document}